\documentclass[11pt]{article}

\usepackage[margin = 0.8in]{geometry}
\usepackage{hyperref, enumerate}
\usepackage{caption,float}
\usepackage{subcaption}
\usepackage{verbatim}
\usepackage{multirow}
\usepackage{graphicx} 
\usepackage{mathrsfs}
\usepackage{comment}
\usepackage{amsthm, amsmath, amssymb, amsfonts, url, booktabs, tikz, setspace, fancyhdr, amsbsy}
\usepackage{tikz, tikz-3dplot, pgf}
\usepackage{longtable}
\usepackage{esint}
\usepackage{epstopdf}
\usepackage{xcolor,epsfig}
\usepackage{parskip}
\usetikzlibrary{arrows}
\usetikzlibrary{decorations}
\usetikzlibrary{decorations.pathmorphing}
\usetikzlibrary{calc}

\newtheorem{theorem}{Theorem}[section]
\newtheorem{proposition}[theorem]{Proposition}
\newtheorem{lemma}[theorem]{Lemma}

\theoremstyle{definition}
\newtheorem{definition}[theorem]{Definition}

\newtheorem{remark}[theorem]{Remark}

\newcommand{\R}{\mathbb{R}}

\newcommand{\V}{\mathbb{V}}
\newcommand{\G}{\mathcal{G}}
\newcommand{\LL}{\mathcal{L}}

\newcommand{\defeq}{\mathrel{\mathop:}=}

\newcommand{\om}{\omega}

\title{Sparse FEONet: A Low-Cost, Memory-Efficient Operator Network via Finite-Element Local Sparsity for Parametric PDEs}

\author{Seungchan Ko\thanks{Department of Mathematics, Inha University, 22212, Incheon, Republic of Korea. Email: \tt{scko@inha.ac.kr}},
~Jiyeon Kim\thanks{Department of Mathematics, Ajou University, 16499, Suwon, Republic of Korea. Email: \tt{gkim0201@ajou.ac.kr}},
~Dongwook Shin\thanks{Department of Mathematics, Ajou University, 16499, Suwon, Republic of Korea. Email: \tt{dws@ajou.ac.kr}}}

\date{}

\begin{document}

\maketitle

\begin{abstract}
    In this paper, we study the finite element operator network (FEONet), an operator-learning method for parametric problems, originally introduced in J. Y. Lee, S. Ko, and Y. Hong, Finite Element Operator Network for Solving Elliptic-Type Parametric PDEs, SIAM J. Sci. Comput., 47(2), C501–C528, 2025. FEONet realizes the parameter-to-solution map on a finite element space and admits a training procedure that does not require training data, while exhibiting high accuracy and robustness across a broad class of problems. However, its computational cost increases and accuracy may deteriorate as the number of elements grows, posing notable challenges for large-scale problems. In this paper, we propose a new sparse network architecture motivated by the structure of finite elements to address this issue. The key idea of our method is to exploit the local properties of finite element basis functions: only neighboring components interact with one another, and this locality is reflected in the design of a sparse neural network. We establish theoretical results demonstrating that the sparse architecture can approximate the target operator effectively and provide a stability analysis ensuring reliable training and prediction.
    Throughout extensive numerical experiments, we show that the proposed sparse network achieves substantial improvements in computational cost and efficiency while maintaining comparable accuracy.
\end{abstract}


\smallskip


\section{Introduction}
The application of machine learning (ML) techniques to partial differential equations (PDEs) has seen remarkable progress in recent years, presenting novel strategies to address persistent difficulties in scientific computing \cite{lagaris1998artificial, PINN001, yu2018deep, ainsworth2021galerkin}. Within this landscape, operator networks have gained attention as an effective and practical approach owing to their capacity to provide rapid solution predictions once training is completed. In contrast to traditional numerical methods that iteratively compute solutions for each new PDE data (e.g., boundary conditions, initial conditions, and external forcing terms), operator networks learn the underlying solution operator from PDE data to the corresponding solutions, thereby enabling fast solution predictions for varying parameters. This advantage positions operator networks as a compelling paradigm for studying parametric PDEs. Representative contributions in this area include the Deep Operator Network (DeepONet) \cite{lu2021learning} and the Fourier Neural Operator (FNO) \cite{li2021fourier}. While both DeepONet and FNO enable fast prediction of solutions under varying PDE data, some challenges remain when extending their use to real-world scenarios. A primary obstacle lies in the necessity of a large collection of pre-computed training datasets of solutions. This is typically generated through classical numerical methods for PDEs, which becomes particularly burdensome for nonlinear or highly complex systems. To address this limitation, hybrid frameworks such as the Physics-Informed DeepONet (PIDeepONet) \cite{wang2021learning} and the Physics-Informed Neural Operator (PINO) \cite{li2021physics} have been proposed. These approaches integrate the benefits of Physics-Informed Neural Networks (PINNs) and operator learning by incorporating governing equations directly into the loss functions of neural operators. Nevertheless, such models still encounter difficulties, including reduced accuracy for intricate geometries, challenges in handling stiff problems, and considerable generalization errors caused by limited training data \cite{jagtap2020extended, costabal2024delta, lee2023hyperdeeponet}. In addition, the use of neural networks to represent solution spaces complicates the enforcement of diverse boundary conditions, which in turn affects the reliability of the resulting solutions \cite{choi2024spectral}.

To overcome these challenges, an unsupervised operator network founded on classical finite element methods (FEMs), referred to as the Finite Element Operator Network (FEONet), has been proposed in \cite{FEONet}. 
Within the FEM framework, the numerical solution $u_h(x)$ with a mesh size $h>0$ is expressed as a linear combination of nodal basis functions $\phi_k(x)$, which are piecewise polynomials defined over meshes. Formally, this can be written as
\begin{equation}\label{FEM_sol}
u_h(x) = \sum_{k=1}^{N_h} \alpha_k \phi_k(x), \quad x \in \Omega,
\end{equation}
where $N_h$ is the number of basis functions depending on $h>0$. Extending this formulation, instead of solving a linear algebraic system to determine $\{\alpha_k\}^{N_h}_{k=1}$ in \eqref{FEM_sol}, FEONet predicts the coefficients $\{\widehat{\alpha}_k\}_{k=1}^{N_h}$ using neural networks to construct the approximate solution for the given PDE. 
In this setting, compared to DeepONet \cite{lu2021learning}, FEONet relies only on a branch network without a trunk network, which significantly reduces the number of parameters. The loss function of FEONet, motivated by classical FEM, is designed using the residual of the Galerkin approximation, which guarantees both accurate PDE solutions and strict enforcement of boundary conditions. Due to the inherent capability of FEM in incorporating boundary conditions, the solutions generated by FEONet precisely satisfy these constraints. 

A distinguishing feature of FEONet is its ability to solve parametric PDEs without requiring paired input-output training data, marking a significant advance in computational efficiency and general applicability. To be more specific, FEONet predicts PDE solutions under varying inputs such as initial and boundary conditions, external forcings, and variable coefficients. As shown in \cite{FEONet}, it demonstrates flexibility in managing diverse PDE instances across complex domains, while avoiding reliance on pre-generated data.   One further advantage of FEONet lies in its versatility, namely, its ability to directly incorporate the techniques from classical numerical analysis. For instance, consider a singularly perturbed problem whose solution develops boundary or interior layers \cite{schlichting2016boundary,batchelor2000introduction}. Such problems are classical yet notoriously difficult in numerical analysis, and many specialized numerical methods have been proposed to treat them \cite{FEM1,FEM2}. One such approach is the enriched finite element method, which employs corrector basis functions derived from asymptotic analysis to capture the stiff behavior of the solution and yield improved FEM approximations \cite{BLA_1,BLA_2,BLA_3}. This idea extends naturally to the FEONet framework: by incorporating a boundary-layer element into the finite element space via an appropriate corrector basis function, one can construct an enriched FEONet basis that effectively captures sharp variations. Further details and a rigorous convergence analysis of FEONet are given in \cite{FEONet, efeonet}.

While FEONet demonstrates flexibility in delivering accurate and efficient solutions across a wide range of scenarios, there still remain some computational challenges. One of the main problems is the issue of computational cost. As previously noted, within the FEONet framework, the dimension of the neural network output coincides with the number of elements employed in the solution representation. For relatively simple problems, we can obtain accurate solutions via FEM with a modest number of bases; however, for complex problems, more refined computations often necessitate a substantially larger number of elements. In such cases, we observed that FEONet’s performance deteriorates as the number of bases grows. A primary cause of this limitation lies in the approximation and generalization issues that arise when the finite element coefficients are represented through neural network approximation. As proved in \cite{feonet_anal}, as the mesh size $h$ decreases (equivalently, as the number of elements increases), the error of FEONet initially decreases, but beyond a certain threshold it begins to increase again. This phenomenon directly corresponds to the key discussion in \cite{feonet_anal}, where the main cause of this phenomenon was identified. 
A second major limitation arises from the scalability issue, which is commonly encountered in large-scale deep-learning tasks. While the errors are mitigated thanks to the theory developed in \cite{feonet_anal}, the computational cost remains substantial for domains with a large number of elements, leading to additional challenges in terms of computation and memory. Problems requiring a large number of elements inevitably suffer from considerable computational overhead, manifested in substantially increased computational cost and impeded training efficiency. This constitutes a crucial challenge from a practical perspective, and it must be addressed if FEONet is to be deployed more broadly in real-world scenarios in a manner comparable to FEM. 

This discussion raises a fundamental question about how to handle computational cost efficiently in operator learning methods, which typically require a large amount of computation. One important idea is to impose a suitable sparsity structure on the neural networks used. If one can design an appropriate sparsity pattern by taking into account the architecture, loss function, and training procedure of a given operator network, then it should be possible to achieve a significant improvement in computational efficiency while maintaining accuracy. We refer to such approaches collectively as Sparse Neural Operator netWorks (SNOW). For widely used architectures such as DeepONet and the FNO, developing SNOW-type methods to effectively control computational cost is, in our view, one of the important future directions in operator learning. 

In this perspective, the main contributions of this paper are as follows.
First, we propose a SNOW approach for FEONet that substantially improves computational efficiency, even when a large number of elements is required.
As we show later, our method is inspired by the observation in FEM that only neighboring elements exert strong interactions, whereas the influence between distant elements is comparatively minor.
Building upon this motivation, we introduce a new strategy using sparse neural networks that enables effective FEONet computations with a significantly reduced number of parameters.
Next, we theoretically establish that the proposed sparse architecture has sufficient approximation capacity for the target operator and admits stable training. In particular, we prove the universal approximation property of the proposed sparse network and provide the stability analysis that guarantees the robust training and solution prediction.
Lastly, we validate the efficiency of the proposed method through a series of numerical experiments. Across a variety of benchmark scenarios, we demonstrate that the proposed approach achieves a substantial reduction in the number of trainable parameters and memory usage while maintaining high accuracy. We also show that the proposed sparse structure outperforms randomly constructed sparse networks in terms of accuracy. In addition, in regimes where very fine meshes are required due to the sharp-transition or high-frequency nature of the solution, the original FEONet becomes practically untrainable, whereas the proposed sparse architecture provides a fast and accurate solution prediction. 

The remainder of the paper is structured as follows. Section \ref{sec:prelim} reviews the preliminaries required for the development of our approach. Section \ref{sec:method} presents the proposed methodology in detail. Section \ref{sec:theory} provides a theoretical analysis supporting the validity of the method, while Section \ref{sec:exp} demonstrates the efficiency of the proposed method through extensive experiments. Finally, Section \ref{sec:conclusion} offers concluding remarks and discusses future research directions.

\section{Preliminaries}\label{sec:prelim}
The objective of this section is to introduce FEONet, which forms the baseline of our proposed sparse methodology. Since FEONet is built upon the classical FEM, we provide a brief overview of the setting of FEM, and subsequently provide a detailed description of FEONet. As a model problem, we shall consider the general second-order linear elliptic PDE of the form
\begin{equation}\label{main_eq}
\begin{aligned}
    -\,{\mathrm{div}}\,(\boldsymbol{a}(x)\nabla u)+\boldsymbol{b}(x)\cdot\nabla u+c(x)u&=f(x)\quad{\mathrm{in}}\,\,D, \\
u(x)&=0\quad\hskip+14pt  {\mathrm{on}}\,\,\partial D.
\end{aligned}
\end{equation}
Here, for simplicity, let us assume the following: 
\begin{equation}\label{coef_cond_1}
\boldsymbol{a}\in L^{\infty}(D)^{d\times d},\quad \boldsymbol{b}\in W^{1,\infty}(D)^d,\quad c\in L^{\infty}(D),\quad f\in H^{-1}(D).
\end{equation}
For the diffusion tensor $\boldsymbol{a}=(a_{ij})$, we further assume uniform ellipticity; that is, there exists a constant $\tilde{a}>0$ such that
\begin{equation}\label{coef_cond_2}
    \sum_{i,j=1}^d a_{ij}(x)\xi_i\xi_j \;\geq\; \tilde{a}\sum_{i=1}^d \xi_i^2,\quad \forall\, \xi=(\xi_1,\ldots,\xi_d)\in\mathbb{R}^d,\; x\in\overline{D}.
\end{equation}
To ensure the well-posedness of the problem \eqref{main_eq}, we additionally assume that
\begin{equation}\label{coef_cond_3}
    c(x)-\tfrac{1}{2}\,\mathrm{div}\,\boldsymbol{b}(x)\;\geq\;0,\quad x\in\overline{D}.
\end{equation}
The weak formulation of the problem is then given as follows: we seek $u\in H^1_0(D)$ such that
\[
B[u,v] := \int_{D}\boldsymbol{a}(x)\nabla u\cdot\nabla v\,\mathrm{d}x 
+ \int_{D}\boldsymbol{b}(x)\cdot\nabla u\, v\,\mathrm{d}x
+ \int_{D}c(x)uv\,\mathrm{d}x= \int_D f(x)v\,\mathrm{d}x =:\ell(v),
\]
for arbitrary $v\in H^1_0(D)$. Under the assumptions \eqref{coef_cond_1}-\eqref{coef_cond_3}, there exist constants $c_0, c_1, c_2>0$ such that
\begin{equation}\label{lax_cond}
B[v,v]\geq c_0\|v\|^2_{H^1(D)}, 
\,\, |B[u,v]|\leq c_1\|u\|_{H^1(D)}\|v\|_{H^1(D)},
\,\, |\ell(v)|\leq c_2\|v\|_{H^1(D)}.
\end{equation}
Therefore, the existence and uniqueness of the weak solution follow directly from the classical Lax--Milgram theorem (see, e.g., \cite{BS_book}). For brevity, the above discussion has been restricted to the linear case. However, as will be explained in detail later, FEONet employs the residual of the variational formulation as the loss functional, and therefore, it can be directly applied to nonlinear problems as well. In fact, in the numerical experiments section, we will also solve some nonlinear equations using the proposed method.

\subsection{Finite element methods}
We first provide a brief overview of the classical FEM. As a preliminary step, we introduce the finite element space that will serve as the foundation throughout the paper.  
Let $\G_h$ denote a shape-regular partition of the physical domain $\overline{D}$, where $h_E$ represents the diameter of an element $E\in\G_h$, and $h=\max_{E\in\G_h}h_E$.  
We assume that there exists a constant $\gamma>0$, independent of $h$, such that
\[
\max_{E\in\G_h}\frac{h_E}{\rho_E}\leq \gamma,
\]
where $\rho_E$ is the supremum of the diameters of inscribed balls in $E\in\G_h$. For a given partition $\G_h$, the corresponding finite element space is defined as
\[
\V_h=\V(\G_h)\;\defeq\;\bigl\{V\in C(\overline{D}):\,V|_E\in\hat{\mathbb{P}}_{\V},\; E\in \G_h,\; V|_{\partial D}=0\bigr\},
\]
where $\hat{\mathbb{P}}_{\V}\subset W^{1,\infty}(\hat{E})$ is a finite-dimensional subspace.  
We further assume that $\V_h$ admits a finite, locally supported basis. Specifically, for each $h>0$, there exists $N_h\in\mathbb{N}$ such that $\V_h=\mathrm{span}\{\phi_1,\ldots,\phi_{N_h}\}$.
For a basis function $\phi_i$, there exists a patch $\mathcal{P}_i$ such that $\textrm{supp}(\phi_i) \subset \mathcal{P}_i$ and $\phi_i(x) = 0$ for $x\not\in \mathcal{P}_i$. 
If $\mathcal{P}_i$ is the set of elements that contain the node $i$, then $\mathcal{P}_i = \textrm{supp}(\phi_i)$.

We now define the Galerkin approximation: we seek a discrete solution of the form
\begin{equation}\label{basis_exp}
u_h=\sum_{k=1}^{N_h}\alpha_k\phi_k\in \V_h    
\end{equation}
satisfying
\begin{equation}\label{var_for}
B[u_h,v_h]=\ell(v_h)\quad \text{for all } v_h\in\V_h.
\end{equation}
Introducing the matrices $A=(A_{ij})_{1\leq i,j\leq N_h}\in\R^{N_h\times N_h}$ and the vector $F=(F_j)_{1\leq j\leq N_h}\in\R^{N_h}$ defined by
\begin{equation}\label{vec_form}
A_{ij}=B[\phi_i,\phi_j]\quad{\rm{and}}
\quad F_j=\ell(\phi_j),
\end{equation}
the discrete scheme \eqref{var_for} can be reformulated as the linear algebraic system
\begin{equation}\label{LAS}
A\alpha=F,\quad \alpha=(\alpha_k)_{1\leq k\leq N_h}\in\R^{N_h}.
\end{equation}
It follows that $B[\phi_i, \phi_j] = 0$ if $\mathcal{P}_i \cap \mathcal{P}_j$ is of measure zero.
As a result, the linear system \eqref{LAS} has a sparse structure.
By solving this system, we determine the coefficient $\alpha$, which in turn allows us to compute an approximate solution of the given PDE \eqref{main_eq} via the basis expansion \eqref{basis_exp}.

\subsection{Finite element operator networks}

We now introduce the finite element operator network (FEONet), originally proposed in \cite{FEONet}.  
As discussed earlier, the input to FEONet can represent various types of PDE data, such as external forces, variable coefficients, or boundary conditions.  
For clarity, we present a prototype setting in which the input corresponds to an external forcing term, though the framework can be naturally extended to other types of inputs. Given a forcing function $f$, rather than computing the coefficients $\alpha$ through the linear system \eqref{LAS}, FEONet predicts them via deep neural networks. Specifically, the neural network takes as input the forcing term $f$, which is parameterized by a random variable $\om$ defined on the probability space $(\Omega,\mathcal{T},\mathbb{P}_{\Omega})$.  
Typical examples include Gaussian random fields or random forcing functions of the form
\[
f(x,\om)=\om_1\sin(2\pi\om_2 x)+\om_3\cos(2\pi\om_4 x),
\]
with $\om=(\om_1,\om_2,\om_3,\om_4)$ drawn i.i.d. uniformly with $\om_i\in[a_i,b_i]$ for $i=1,2,3,4$. Once the input feature $\om\in\Omega$ passes through the neural network, it outputs coefficients $\{\widehat{\alpha}_i\}_{i=1}^{N_h}$, and the solution prediction is reconstructed as
\begin{equation}\label{sol_recon}
\widehat{u}_h(x,\om)=\sum_{i=1}^{N_h}\widehat{\alpha}_i(\om)\phi_i(x).
\end{equation}

For training, the population loss is defined by the residual of the variational formulation \eqref{var_for}:
\begin{equation}\label{pop_loss}
\LL(\alpha)=\mathbb{E}_{\om\sim\mathbb{P}_{\Omega}}
\left[\sum_{i=1}^{N_h}\big|B[\widehat{u}_h(x,\om),\phi_i(x)]-\ell(\phi_i(x))\big|^2\right]^{1/2}.
\end{equation}
In practice, the empirical loss function is employed, obtained via Monte Carlo sampling of \eqref{pop_loss}:
\begin{equation}\label{emp_loss}
\LL^M(\alpha)=\frac{|\Omega|}{M}\sum_{j=1}^M
\left[\sum_{i=1}^{N_h}\big|B[\widehat{u}_h(x,\om_j),\phi_i(x)]-\ell(\phi_i(x))\big|^2\right]^{1/2},
\end{equation}
where $\{\om_j\}_{j=1}^M$ are i.i.d. samples drawn from $\mathbb{P}_{\Omega}$.  
At each training epoch, the network parameters are updated to minimize $\LL^M$, until the empirical loss becomes sufficiently small. The final prediction is then obtained via \eqref{sol_recon}, after sufficient training has been performed.  

\begin{figure}[t]
    \centering
    \includegraphics[width=0.9\textwidth]{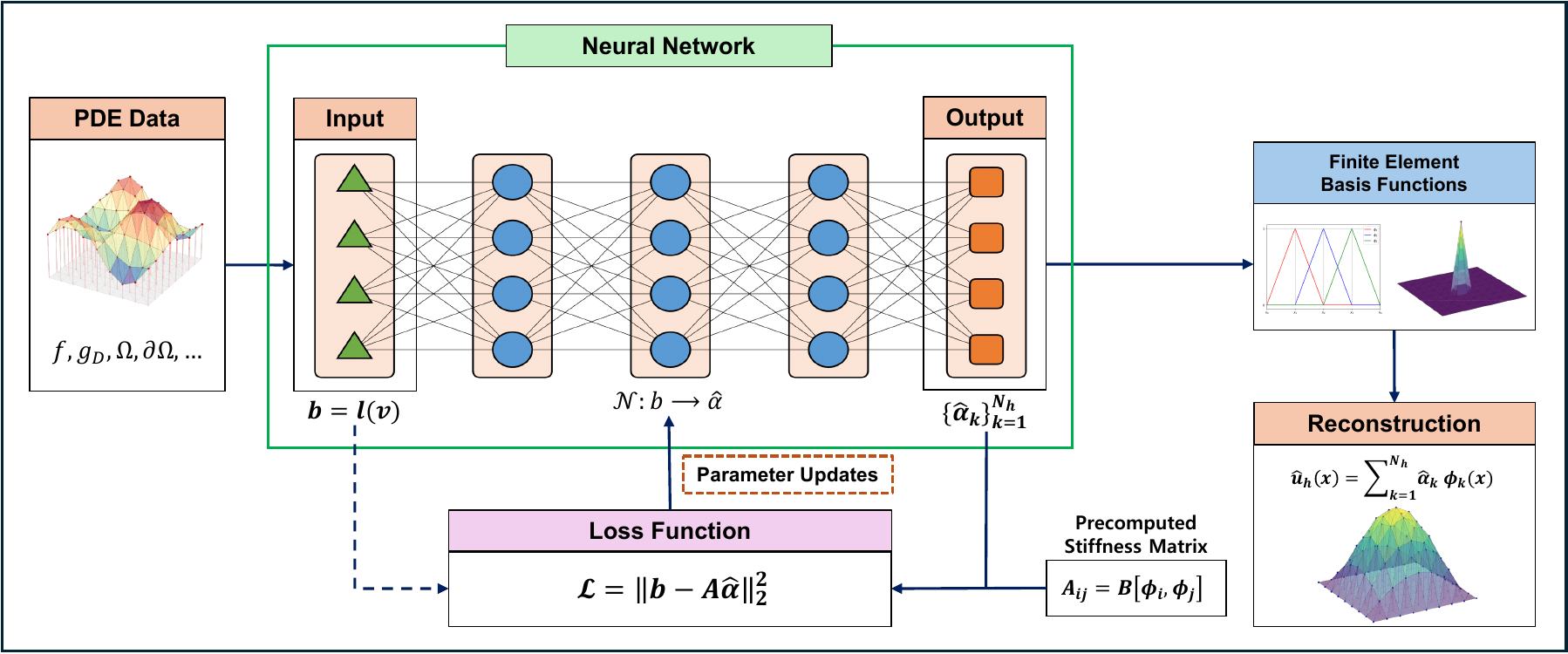}
    \caption{
    Schematic overview of FEONet.
    }
    \label{fig:feonet_architecture}
\end{figure}

A crucial feature of FEONet is that training relies solely on random samples from $\Omega$, without requiring any precomputed input-output pairs.  
Moreover, since the method is based on the basis expansion \eqref{sol_recon}, exact boundary conditions can be imposed in the same way as in classical FEM.  
A schematic overview of the FEONet structure is given in Figure \ref{fig:feonet_architecture}. Extensive numerical experiments on benchmark problems are reported in \cite{FEONet}, confirming the effectiveness of the approach in terms of accuracy, generalization capability, and computational efficiency, and the rigorous convergence analysis of FEONet was conducted in \cite{feonet_anal}. As noted in these papers, a major limitation of FEONet is that its computational cost grows rapidly as the number of finite elements increases, which makes large-scale applications challenging. In some cases, training may not progress at all. As discussed before, this observation provides a key motivation for our work. To address the computational cost, training efficiency, and memory bottlenecks, we propose in this paper a new sparse architecture tailored to FEONet. As a simple motivating example, Figure \ref{fig:loss_comparison} summarizes learning curves for different values of $N_h$, the number of degrees of freedom in the finite element discretization. When $N_h$ is small, the original FEONet and our proposed method exhibit comparable behavior. However, for relatively large $N_h$, training of the original FEONet often fails to make sufficient progress and is terminated early (e.g., via early stopping), whereas our sparse architecture consistently yields robust training performance across all tested values of $N_h$. As will be further demonstrated in the experiments, many practical settings require fine meshes with a large number of elements; therefore, our method is expected to advance FEONet toward scalable, real-world applicability.

\begin{figure}[!t]
    \centering
    \includegraphics[width=\textwidth]{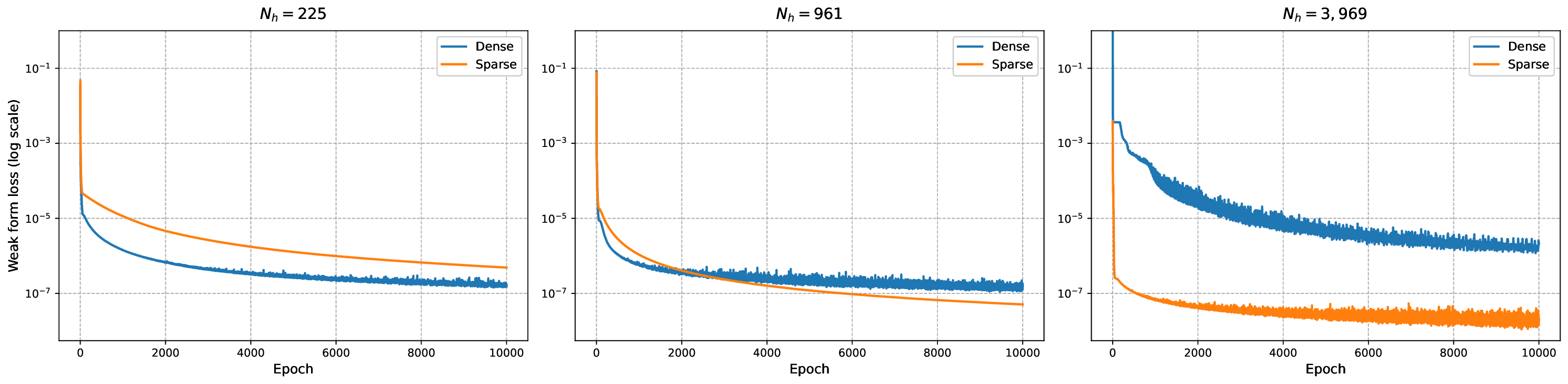}
    \caption{
    Comparison of training losses between fully-connected (dense) and sparsely-connected (sparse) FEONet for the 2D advection-diffusion-reaction equation across different resolutions $N_h=225, 961, \text{and} \ 3{,}969$. As $N_h$ increases, the sparse network shows more stable convergence than the dense connected network.
    }
    \label{fig:loss_comparison}
\end{figure}

\section{Methodology: Sparse FEONet}\label{sec:method}
As in \cite{FEONet}, FEONet can take fully-connected (FC) neural networks or convolutional neural networks (CNN) as model architectures, but in this paper, we shall consider the FC case. 
In an FC layer with $N$ input nodes and $M$ output nodes, the number of weights and biases equals $(N+1)M$.
In this case, the dimension of the parameter space grows rapidly as the number of nodes and layers increases.
Thus, from an optimization viewpoint, it is crucial to reduce the number of parameters.
This motivates us to introduce a new approach that assigns sparse weights to each layer.
In this approach, we represent each layer’s weight matrix in a sparse-matrix format.

\begin{figure}[t!]
\begin{center}
\begin{tikzpicture}[scale=0.75, transform shape]

\draw[fill=yellow!15] (-0.5,0.5)--(3.5,0.5)--(3.5,4.5)--(-0.5,4.5)--(-0.5,0.5);

\draw[fill=red!20] (-0.5,1.5)--(0.5,0.5)--(1.5,0.5)--(1.5,1.5)--(0.5,2.5)--(-0.5,2.5)--(-0.5,1.5);
\draw[fill=blue!20] (3.5,4.5)--(3.5,3.5)--(2.5,4.5)--(3.5,4.5);

\foreach \x in {-1,0,1,2,3,4} {
\draw[thick] (-1.5, 1.5+\x) -- (0.5+\x, -0.5);
\draw[thick] (4.5, 1.5+\x) -- (0.5+\x, 5.5);
  \foreach \y in {-1,0,1,2,3,4,5} {
    \draw[thick] (\x-0.5,\y+0.5) -- (\x+0.5,\y+0.5);
  }
}

\foreach \x in {-1,0,1,2,3,4,5} {
  \foreach \y in {-1,0,1,2,3,4} {
    \draw[thick] (\x-0.5,\y+0.5) -- (\x-0.5,\y+1.5);
  }
}

\foreach \x in {0,1,2,3,4} {
  \fill[draw=black, thick, fill=white] (-1.5,\x+0.5) circle (4pt);
  \fill[draw=black, thick, fill=white] (4.5,\x+0.5) circle (4pt);
  \fill[draw=black, thick, fill=white] (\x-0.5,-0.5) circle (4pt);
  \fill[draw=black, thick, fill=white] (\x-0.5,5.5) circle (4pt);
  \foreach \y in {0,1,2,3,4} {
    \fill[black!100] (\x-0.5,\y+0.5) circle (4pt);  
  }
  \draw (\x-0.75, 0.45) node[below]{\the\numexpr \x+1\relax};
  \draw (\x-0.75, 1.45) node[below]{\the\numexpr \x+6\relax};
  \draw (\x-0.75, 2.45) node[below]{\the\numexpr \x+11\relax};
  \draw (\x-0.75, 3.45) node[below]{\the\numexpr \x+16\relax};
  \draw (\x-0.75, 4.45) node[below]{\the\numexpr \x+21\relax};
}
\fill[draw=black, thick, fill=white] (-1.5,-0.5) circle (4pt);
\fill[draw=black, thick, fill=white] (-1.5,5.5) circle (4pt);
\fill[draw=black, thick, fill=white] (4.5,-0.5) circle (4pt);
\fill[draw=black, thick, fill=white] (4.5,5.5) circle (4pt);

\draw[red, fill=red] (0.5,1.5) circle [radius = 4pt];
\draw[blue, fill=blue] (3.5,4.5) circle [radius = 4pt];

\foreach \x in {1,2,...,25} {
\draw[thick, fill=black!50] (7,-3+\x*0.4) circle [radius=0.1] node[left]{\the\numexpr 26-\x\relax};
\draw[thick, fill=black!50] (10,-3+\x*0.4) circle [radius=0.1] node[right]{\the\numexpr 26-\x\relax};
    \foreach \y in {1,2,...,25}{
    \draw[ultra thin, black!50] (7.1, -3+\x*0.4) -- (9.9, -3+\y*0.4);
    }
}

\foreach \x in {1,2,...,25} {
\draw[thick, fill=black!50] (13,-3+\x*0.4) circle [radius=0.1] node[left]{\the\numexpr 26-\x\relax};
\draw[thick, fill=black!50] (16,-3+\x*0.4) circle [radius=0.1] node[right]{\the\numexpr 26-\x\relax};
\draw[ultra thin, black!50] (13.1,-3+\x*0.4)--(15.9,-3+\x*0.4);
}

\foreach \y in {0,1,2,3}{
    \draw[ultra thin, black!50] (13.1, -2.6+\y*2)--(15.9,-2.2+\y*2);
    \draw[ultra thin, black!50] (13.1, -2.6+\y*2)--(15.9,-0.6+\y*2);
    \draw[ultra thin, black!50] (13.1, -0.6+\y*2)--(15.9,-2.2+\y*2);
    \draw[ultra thin, black!50] (13.1, -0.6+\y*2)--(15.9,-2.6+\y*2);

    \draw[ultra thin, black!50] (13.1, -1+\y*2)--(15.9,-1.4+\y*2);
    \draw[ultra thin, black!50] (13.1, -1+\y*2)--(15.9, 1+\y*2);
    \draw[ultra thin, black!50] (13.1, -1+\y*2)--(15.9,0.6+\y*2);
    \draw[ultra thin, black!50] (13.1, 1+\y*2)--(15.9, -1+\y*2);

    \draw[ultra thin, black!50] (13.1, 7.0-\y*0.4)--(15.9,6.6-\y*0.4);
    \draw[ultra thin, black!50] (13.1, 6.6-\y*0.4)--(15.9,7.0-\y*0.4);
    
    \foreach \x in {1,2,3}{
    \draw[ultra thin, black!50] (13.1, -2.6+2*\y+\x*0.4)--(15.9,-2.2+2*\y+\x*0.4);
    \draw[ultra thin, black!50] (13.1, -2.6+2*\y+\x*0.4)--(15.9,-3+2*\y+\x*0.4);
    \draw[ultra thin, black!50] (13.1, -2.6+2*\y+\x*0.4)--(15.9,-1+2*\y+\x*0.4);

    \draw[ultra thin, black!50] (13.1, -0.6+2*\y+\x*0.4)--(15.9,-2.6+2*\y+\x*0.4);
    \draw[ultra thin, black!50] (13.1, -2.6+2*\y+\x*0.4)--(15.9,-0.6+2*\y+\x*0.4);
    \draw[ultra thin, black!50] (13.1, -0.6+2*\y+\x*0.4)--(15.9,-2.2+2*\y+\x*0.4);
    }
}

\draw[blue, fill=blue] (7,-2.6) circle [radius=0.1];
\draw[red, fill=red] (7,4.6) circle [radius=0.1];

\draw[blue, fill=blue] (13,-2.6) circle [radius=0.1];
\draw[red, fill=red] (13,4.6) circle [radius=0.1];

\foreach \y in {1,2,...,25}{
    \draw[blue, thick] (7.1, -2.6) -- (9.9, -3+\y*0.4);
    \draw[red, thick] (7.1, 4.6) -- (9.9, -3+\y*0.4);
}

\draw[blue, thick] (13.1, -2.6) -- (15.9, -2.6);
\draw[blue, thick] (13.1, -2.6) -- (15.9, -2.2);
\draw[blue, thick] (13.1, -2.6) -- (15.9, -0.6);

\draw[red, thick] (13.1, 4.6) -- (15.9, 4.6);
\draw[red, thick] (13.1, 4.6) -- (15.9, 4.2);
\draw[red, thick] (13.1, 4.6) -- (15.9, 5);
\draw[red, thick] (13.1, 4.6) -- (15.9, 6.6);
\draw[red, thick] (13.1, 4.6) -- (15.9, 6.2);
\draw[red, thick] (13.1, 4.6) -- (15.9, 2.6);
\draw[red, thick] (13.1, 4.6) -- (15.9, 3);

\draw (8.5,7.5) node{Fully Connected Layer};
\draw (8.5,-3.2) node{(\# of weights: 625)};
\draw (14.5,7.5) node{Sparse Layer ($C_\ell= 1$)};
\draw (14.5,-3.2) node{(\# of weights: 137)};
\end{tikzpicture}  

\caption{Comparison of weight connectivity in fully-connected and sparse layers ($C_\ell = 1$).} \label{sparse_network}
\end{center}    
\end{figure}
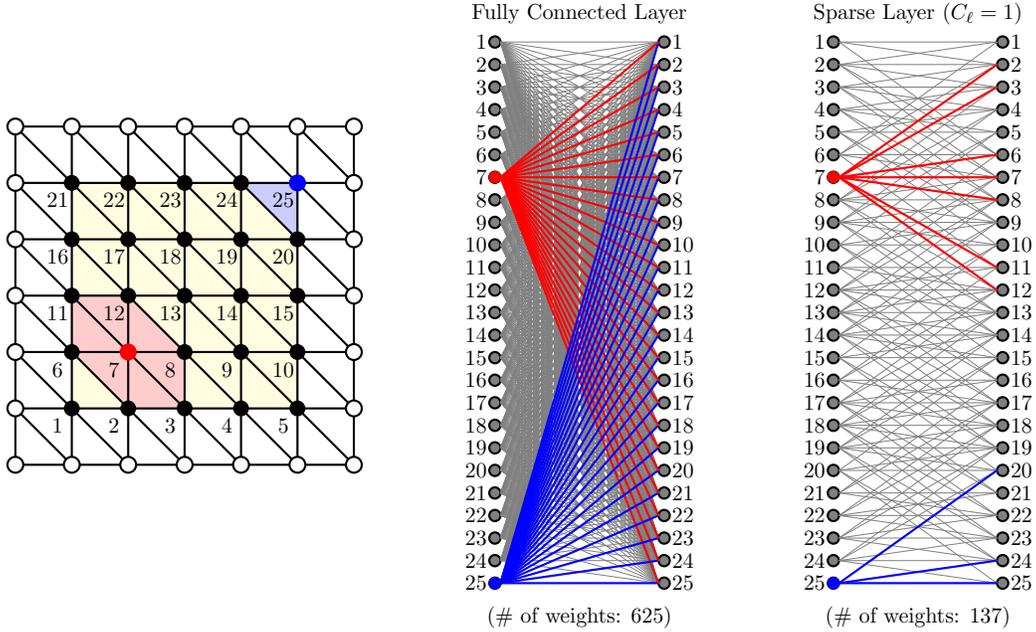

We now describe how to design a sparse network for the FEONet. 
To clearly convey the idea, we shall consider the unit-square domain, and the case of homogeneous Dirichlet boundary condition as in \eqref{main_eq}, which allows us to consider only the interior nodes. 
In Figure~\ref{sparse_network}, we need to determine the coefficients $\{\alpha_k\}$ in \eqref{basis_exp} for the 25 interior nodes. 
Hence, the input layer has 25 nodes, which matches the number of unknowns. 
For simplicity, we also set 25 nodes in each hidden layer and in the output layer.
The parameter $C_\ell$ determines the level of support expansion for the basis function associated with each node. The key idea is that, in the FEM, interactions occur only between neighboring elements. In particular, $B[\phi_i, \phi_j] = 0$ whenever the support of $\phi_i$ and $\phi_j$ do not overlap. Accordingly, we design the neural network to be sparse by introducing connectivity only between nodes corresponding to elements that influence each other (i.e., adjacent elements). More precisely, for $C_\ell = 1$, we connect a node only to the nodes contained in the support of its basis function.
For $C_\ell = 2$, new nodes are further connected to all nodes contained in the union of the supports of the nodes selected at $C_\ell = 1$. 
Thus, the number of connections increases as $C_\ell$ increases. 
Figure~\ref{sparse_network} shows the FC layer and the sparse layer when $C_\ell = 1$.
For example, the nodes contained in the support of the basis function at node 7 are nodes 2, 3, 6, 7, 8, 11, and 12. We therefore connect input node 7 to those output nodes. To quantify the degree of sparsity more precisely, we introduce the notion of sparsity measure.
With this sparse connectivity above, the number of weights is 137 with the sparsity 0.7808, where
$$S = 1 - \frac{\textrm{\# of (nonzero) weights for a sparse layer }}{\textrm{\# of weights for a FC layer}}.$$
Table~\ref{tab:FC_vs_Sparse} shows the number of weights and the sparsity $S$ for the FC layer and for sparse layers at several values of $C_\ell$.
For fixed $N_h$, the sparsity decreases as $C_\ell$ increases; when $C_\ell$ is large enough, the sparse layer coincides with the FC layer, as in the case $N_h = 25$.
As $N_h$ increases, the sparsity increases, and for sufficiently large $N_h$, the sparsity becomes less sensitive to $C_\ell$, which makes our proposed sparsity more practical.

\begin{table}
\caption{Comparison of the number of weights for FC and sparse layers (parentheses indicate sparsity $S$).} \label{tab:FC_vs_Sparse}
\begin{center}
    \begin{tabular}{|r||r|r|r|r|r|}
        \hline
        \multirow{2}{*}{$N_h$ \hspace*{3pt}} & \multirow{2}{*}{\centering FC layer \hspace*{1.5pt}} & \multicolumn{4}{c|}{Sparse layer} \\
        \cline{3-6}
        & & \multicolumn{1}{c|}{$C_\ell = 1$} & \multicolumn{1}{c|}{$C_\ell = 4$} & \multicolumn{1}{c|}{$C_\ell = 8$} & \multicolumn{1}{c|}{$C_\ell = 15$} \\
         \hline
         \hline
         25 & 625 & 137 (0.7808) & 555 (0.1120) & 625 (0.0000) & \multicolumn{1}{c|}{--} \\
         \hline
         100 & 10,000 & 622 (0.9378) & 3,930 (0.6070) & 8,392 (0.1608) & 9,970 (0.0030) \\
         \hline
         900 & 810,000 & 6,062 (0.9925) & 47,930 (0.9408) & 149,352 (0.8156) & 384,860 (0.5249) \\
         \hline
         2,500 & 6,250,000 & 17,102 (0.9973) & 140,730 (0.9775) & 463,912 (0.9258) & 1,340,060 (0.7856) \\
         \hline
         10,000 & 100,000,000 & 69,202 (0.9993) & 586,230 (0.9941) & 2,009,812 (0.9799) & 6,251,560 (0.9375) \\
         \hline
    \end{tabular}
\end{center}
\end{table}

To provide a more formal illustration of the proposed method, we consider the following representative example. For simplicity, let $\Omega = (0,1)^2 \subset \mathbb{R}^2$ be the unit square. Note that the same construction extends to general dimensions and more general domains. We consider a uniform Cartesian grid with $h = 1/n$ in both $x$ and $y$ axis directions, with integer indices $i,\ j \in \{0, 1, \dots, n\}.$ On each grid square, we take the standard right isosceles triangle split (e.g., cut along a diagonal), yielding a triangulation of $\Omega$ for a conforming piecewise linear finite element method. Since we consider homogeneous Dirichlet boundary conditions, only the values corresponding to interior nodes $(i,j)$ with $1\le i, \ j \le n-1$ are unknowns.
Here, we assign a single global node number by the row-major mapping $k = (n-1)(j-1) + i$ (see, e.g., Figure~\ref{sparse_network}). Let $\eta_{i,j} = (ih, jh)$ denote the coordinate of node $(i, j)$ and $\mathcal{V}_h = \{(i,j): 1\le i, \ j\le n-1\}$ be the set of interior nodes with the size $N_h = |\mathcal{V}_h| = (n-1)^2$. Also, we denote the corresponding nodal basis centered at $\eta_{i,j}$ by $\phi_{i,j}$. For the global indices $k$ and $\ell$, we may write $\phi_k = \phi_{p,q}$ and $\phi_\ell = \phi_{r,s}$ with $p,q,r,s \in \{1,\dots,n-1\}$. Under this notation, \eqref{vec_form} can be rewritten as $A_{k\ell} = 
A_{(p,q),(r,s)}$ and $F_k = F_{(p,q)}$.

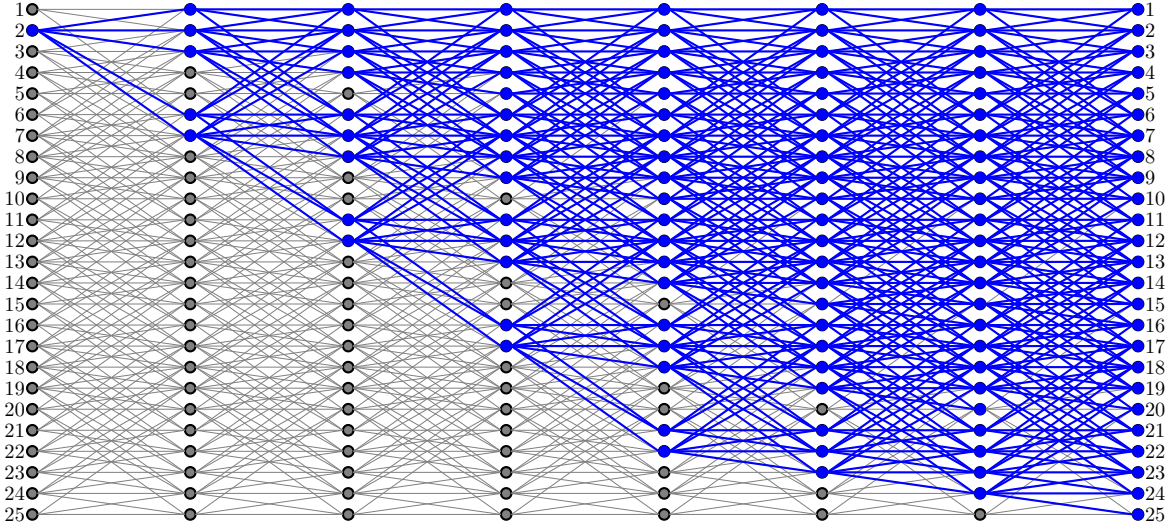
\begin{figure}[t!]
\begin{center}
\begin{tikzpicture}[scale=0.7, transform shape]

\foreach \x in {1,2,...,25} {
\draw[thick, fill=black!50] (0,-3+\x*0.4) circle [radius=0.1] node[left]{\the\numexpr 26-\x\relax};
\draw (21,-3+\x*0.4) node[right]{\the\numexpr 26-\x\relax};
}
\foreach \k in {1,2,3,4,5,6,7} {
\foreach \x in {1,2,...,25} {
\draw[thick, fill=black!50] (3*\k,-3+\x*0.4) circle [radius=0.1];
\draw[ultra thin, black!50] (3*\k-2.9,-3+\x*0.4)--(3*\k-0.1,-3+\x*0.4);
}
}

\foreach \k in {1,2,3,4,5,6,7} {
\foreach \y in {0,1,2,3}{
    \draw[ultra thin, black!50] (3*\k-2.9, -2.6+\y*2)--(3*\k-0.1,-2.2+\y*2);
    \draw[ultra thin, black!50] (3*\k-2.9, -2.6+\y*2)--(3*\k-0.1,-0.6+\y*2);
    \draw[ultra thin, black!50] (3*\k-2.9, -0.6+\y*2)--(3*\k-0.1,-2.2+\y*2);
    \draw[ultra thin, black!50] (3*\k-2.9, -0.6+\y*2)--(3*\k-0.1,-2.6+\y*2);

    \draw[ultra thin, black!50] (3*\k-2.9, -1+\y*2)--(3*\k-0.1,-1.4+\y*2);
    \draw[ultra thin, black!50] (3*\k-2.9, -1+\y*2)--(3*\k-0.1, 1+\y*2);
    \draw[ultra thin, black!50] (3*\k-2.9, -1+\y*2)--(3*\k-0.1,0.6+\y*2);
    \draw[ultra thin, black!50] (3*\k-2.9, 1+\y*2)--(3*\k-0.1, -1+\y*2);

    \draw[ultra thin, black!50] (3*\k-2.9, 7.0-\y*0.4)--(3*\k-0.1,6.6-\y*0.4);
    \draw[ultra thin, black!50] (3*\k-2.9, 6.6-\y*0.4)--(3*\k-0.1,7.0-\y*0.4);
    
    \foreach \x in {1,2,3}{
    \draw[ultra thin, black!50] (3*\k-2.9, -2.6+2*\y+\x*0.4)--(3*\k-0.1,-2.2+2*\y+\x*0.4);
    \draw[ultra thin, black!50] (3*\k-2.9, -2.6+2*\y+\x*0.4)--(3*\k-0.1,-3+2*\y+\x*0.4);
    \draw[ultra thin, black!50] (3*\k-2.9, -2.6+2*\y+\x*0.4)--(3*\k-0.1,-1+2*\y+\x*0.4);

    \draw[ultra thin, black!50] (3*\k-2.9, -0.6+2*\y+\x*0.4)--(3*\k-0.1,-2.6+2*\y+\x*0.4);
    \draw[ultra thin, black!50] (3*\k-2.9, -2.6+2*\y+\x*0.4)--(3*\k-0.1,-0.6+2*\y+\x*0.4);
    \draw[ultra thin, black!50] (3*\k-2.9, -0.6+2*\y+\x*0.4)--(3*\k-0.1,-2.2+2*\y+\x*0.4);
    }
}
}

\draw[blue, fill=blue] (0,6.6) circle [radius=0.1];

\foreach \k in {1,...,7} {
\draw[blue, thick] (3*\k-2.9, 6.6) -- (3*\k-0.1, 7); 
\draw[blue, thick] (3*\k-2.9, 6.6) -- (3*\k-0.1, 6.6);
\draw[blue, thick] (3*\k-2.9, 6.6) -- (3*\k-0.1, 6.2);
\draw[blue, thick] (3*\k-2.9, 6.6) -- (3*\k-0.1, 5); 
\draw[blue, thick] (3*\k-2.9, 6.6) -- (3*\k-0.1, 4.6);

\draw[blue, fill=blue] (3*\k,7) circle [radius=0.1]; 
\draw[blue, fill=blue] (3*\k,6.6) circle [radius=0.1];
\draw[blue, fill=blue] (3*\k,6.2) circle [radius=0.1];
\draw[blue, fill=blue] (3*\k,5) circle [radius=0.1]; 
\draw[blue, fill=blue] (3*\k,4.6) circle [radius=0.1];
}

\foreach \k in {2,...,7} {
\draw[blue, thick] (3*\k-2.9, 7) -- (3*\k-0.1, 7); 
\draw[blue, thick] (3*\k-2.9, 7) -- (3*\k-0.1, 6.6); 
\draw[blue, thick] (3*\k-2.9, 7) -- (3*\k-0.1, 5); 
\draw[blue, thick] (3*\k-2.9, 6.2) -- (3*\k-0.1, 6.6); 
\draw[blue, thick] (3*\k-2.9, 6.2) -- (3*\k-0.1, 6.2); 
\draw[blue, thick] (3*\k-2.9, 6.2) -- (3*\k-0.1, 5.8); 
\draw[blue, thick] (3*\k-2.9, 6.2) -- (3*\k-0.1, 4.6); 
\draw[blue, thick] (3*\k-2.9, 6.2) -- (3*\k-0.1, 4.2); 

\draw[blue, thick] (3*\k-2.9, 5) -- (3*\k-0.1, 7); 
\draw[blue, thick] (3*\k-2.9, 5) -- (3*\k-0.1, 6.6); 
\draw[blue, thick] (3*\k-2.9, 5) -- (3*\k-0.1, 5); 
\draw[blue, thick] (3*\k-2.9, 5) -- (3*\k-0.1, 4.6); 
\draw[blue, thick] (3*\k-2.9, 5) -- (3*\k-0.1, 3); 

\draw[blue, thick] (3*\k-2.9, 4.6) -- (3*\k-0.1, 6.6); 
\draw[blue, thick] (3*\k-2.9, 4.6) -- (3*\k-0.1, 6.2); 
\draw[blue, thick] (3*\k-2.9, 4.6) -- (3*\k-0.1, 5); 
\draw[blue, thick] (3*\k-2.9, 4.6) -- (3*\k-0.1, 4.6); 
\draw[blue, thick] (3*\k-2.9, 4.6) -- (3*\k-0.1, 4.2); 
\draw[blue, thick] (3*\k-2.9, 4.6) -- (3*\k-0.1, 3); 
\draw[blue, thick] (3*\k-2.9, 4.6) -- (3*\k-0.1, 2.6); 

\draw[blue, fill=blue] (3*\k,5.8) circle [radius=0.1]; 
\draw[blue, fill=blue] (3*\k,4.2) circle [radius=0.1]; 
\draw[blue, fill=blue] (3*\k,3) circle [radius=0.1]; 
\draw[blue, fill=blue] (3*\k,2.6) circle [radius=0.1]; 
}

\foreach \k in {3,...,7} {
\draw[blue, thick] (3*\k-2.9, 5.8) -- (3*\k-0.1, 6.2); 
\draw[blue, thick] (3*\k-2.9, 5.8) -- (3*\k-0.1, 5.8); 
\draw[blue, thick] (3*\k-2.9, 5.8) -- (3*\k-0.1, 5.4); 
\draw[blue, thick] (3*\k-2.9, 5.8) -- (3*\k-0.1, 4.2); 
\draw[blue, thick] (3*\k-2.9, 5.8) -- (3*\k-0.1, 3.8); 

\draw[blue, thick] (3*\k-2.9, 4.2) -- (3*\k-0.1, 6.2); 
\draw[blue, thick] (3*\k-2.9, 4.2) -- (3*\k-0.1, 5.8); 
\draw[blue, thick] (3*\k-2.9, 4.2) -- (3*\k-0.1, 4.6); 
\draw[blue, thick] (3*\k-2.9, 4.2) -- (3*\k-0.1, 4.2); 
\draw[blue, thick] (3*\k-2.9, 4.2) -- (3*\k-0.1, 3.8); 
\draw[blue, thick] (3*\k-2.9, 4.2) -- (3*\k-0.1, 2.6); 
\draw[blue, thick] (3*\k-2.9, 4.2) -- (3*\k-0.1, 2.2); 

\draw[blue, thick] (3*\k-2.9, 3) -- (3*\k-0.1, 5); 
\draw[blue, thick] (3*\k-2.9, 3) -- (3*\k-0.1, 4.6); 
\draw[blue, thick] (3*\k-2.9, 3) -- (3*\k-0.1, 3); 
\draw[blue, thick] (3*\k-2.9, 3) -- (3*\k-0.1, 2.6); 
\draw[blue, thick] (3*\k-2.9, 3) -- (3*\k-0.1, 1); 

\draw[blue, thick] (3*\k-2.9, 2.6) -- (3*\k-0.1, 4.6); 
\draw[blue, thick] (3*\k-2.9, 2.6) -- (3*\k-0.1, 4.2); 
\draw[blue, thick] (3*\k-2.9, 2.6) -- (3*\k-0.1, 3); 
\draw[blue, thick] (3*\k-2.9, 2.6) -- (3*\k-0.1, 2.6); 
\draw[blue, thick] (3*\k-2.9, 2.6) -- (3*\k-0.1, 2.2); 
\draw[blue, thick] (3*\k-2.9, 2.6) -- (3*\k-0.1, 1); 
\draw[blue, thick] (3*\k-2.9, 2.6) -- (3*\k-0.1, 0.6); 

\draw[blue, fill=blue] (3*\k,5.4) circle [radius=0.1]; 
\draw[blue, fill=blue] (3*\k,3.8) circle [radius=0.1]; 
\draw[blue, fill=blue] (3*\k,2.2) circle [radius=0.1]; 
\draw[blue, fill=blue] (3*\k,1) circle [radius=0.1]; 
\draw[blue, fill=blue] (3*\k,0.6) circle [radius=0.1]; 
}

\foreach \k in {4,5,6,7} {
\draw[blue, thick] (3*\k-2.9, 5.4) -- (3*\k-0.1, 5.8); 
\draw[blue, thick] (3*\k-2.9, 5.4) -- (3*\k-0.1, 5.4); 
\draw[blue, thick] (3*\k-2.9, 5.4) -- (3*\k-0.1, 3.8); 
\draw[blue, thick] (3*\k-2.9, 5.4) -- (3*\k-0.1, 3.4); 

\draw[blue, thick] (3*\k-2.9, 3.8) -- (3*\k-0.1, 5.8); 
\draw[blue, thick] (3*\k-2.9, 3.8) -- (3*\k-0.1, 5.4); 
\draw[blue, thick] (3*\k-2.9, 3.8) -- (3*\k-0.1, 4.2); 
\draw[blue, thick] (3*\k-2.9, 3.8) -- (3*\k-0.1, 3.8); 
\draw[blue, thick] (3*\k-2.9, 3.8) -- (3*\k-0.1, 3.4); 
\draw[blue, thick] (3*\k-2.9, 3.8) -- (3*\k-0.1, 2.2); 
\draw[blue, thick] (3*\k-2.9, 3.8) -- (3*\k-0.1, 1.8); 

\draw[blue, thick] (3*\k-2.9, 2.2) -- (3*\k-0.1, 4.2); 
\draw[blue, thick] (3*\k-2.9, 2.2) -- (3*\k-0.1, 3.8); 
\draw[blue, thick] (3*\k-2.9, 2.2) -- (3*\k-0.1, 2.6); 
\draw[blue, thick] (3*\k-2.9, 2.2) -- (3*\k-0.1, 2.2); 
\draw[blue, thick] (3*\k-2.9, 2.2) -- (3*\k-0.1, 1.8); 
\draw[blue, thick] (3*\k-2.9, 2.2) -- (3*\k-0.1, 0.6); 
\draw[blue, thick] (3*\k-2.9, 2.2) -- (3*\k-0.1, 0.2); 

\draw[blue, thick] (3*\k-2.9, 1) -- (3*\k-0.1, 3); 
\draw[blue, thick] (3*\k-2.9, 1) -- (3*\k-0.1, 2.6); 
\draw[blue, thick] (3*\k-2.9, 1) -- (3*\k-0.1, 1); 
\draw[blue, thick] (3*\k-2.9, 1) -- (3*\k-0.1, 0.6); 
\draw[blue, thick] (3*\k-2.9, 1) -- (3*\k-0.1, -1); 

\draw[blue, thick] (3*\k-2.9, 0.6) -- (3*\k-0.1, 2.6); 
\draw[blue, thick] (3*\k-2.9, 0.6) -- (3*\k-0.1, 2.2); 
\draw[blue, thick] (3*\k-2.9, 0.6) -- (3*\k-0.1, 1); 
\draw[blue, thick] (3*\k-2.9, 0.6) -- (3*\k-0.1, 0.6); 
\draw[blue, thick] (3*\k-2.9, 0.6) -- (3*\k-0.1, 0.2); 
\draw[blue, thick] (3*\k-2.9, 0.6) -- (3*\k-0.1, -1); 
\draw[blue, thick] (3*\k-2.9, 0.6) -- (3*\k-0.1, -1.4); 

\draw[blue, fill=blue] (3*\k,3.4) circle [radius=0.1]; 
\draw[blue, fill=blue] (3*\k,1.8) circle [radius=0.1]; 
\draw[blue, fill=blue] (3*\k,0.2) circle [radius=0.1]; 
\draw[blue, fill=blue] (3*\k,-1) circle [radius=0.1]; 
\draw[blue, fill=blue] (3*\k,-1.4) circle [radius=0.1]; 
}

\foreach \k in {5,6,7} {
\draw[blue, thick] (3*\k-2.9, 3.4) -- (3*\k-0.1, 5.4); 
\draw[blue, thick] (3*\k-2.9, 3.4) -- (3*\k-0.1, 3.8); 
\draw[blue, thick] (3*\k-2.9, 3.4) -- (3*\k-0.1, 3.4); 
\draw[blue, thick] (3*\k-2.9, 3.4) -- (3*\k-0.1, 1.8); 
\draw[blue, thick] (3*\k-2.9, 3.4) -- (3*\k-0.1, 1.4); 

\draw[blue, thick] (3*\k-2.9, 1.8) -- (3*\k-0.1, 3.8); 
\draw[blue, thick] (3*\k-2.9, 1.8) -- (3*\k-0.1, 3.4); 
\draw[blue, thick] (3*\k-2.9, 1.8) -- (3*\k-0.1, 2.2); 
\draw[blue, thick] (3*\k-2.9, 1.8) -- (3*\k-0.1, 1.8); 
\draw[blue, thick] (3*\k-2.9, 1.8) -- (3*\k-0.1, 1.4); 
\draw[blue, thick] (3*\k-2.9, 1.8) -- (3*\k-0.1, 0.2); 
\draw[blue, thick] (3*\k-2.9, 1.8) -- (3*\k-0.1, -0.2);

\draw[blue, thick] (3*\k-2.9, .2) -- (3*\k-0.1, 2.2); 
\draw[blue, thick] (3*\k-2.9, .2) -- (3*\k-0.1, 1.8); 
\draw[blue, thick] (3*\k-2.9, .2) -- (3*\k-0.1, 0.6); 
\draw[blue, thick] (3*\k-2.9, .2) -- (3*\k-0.1, 0.2); 
\draw[blue, thick] (3*\k-2.9, .2) -- (3*\k-0.1, -0.2); 
\draw[blue, thick] (3*\k-2.9, .2) -- (3*\k-0.1, -1.4); 
\draw[blue, thick] (3*\k-2.9, .2) -- (3*\k-0.1, -1.8); 

\draw[blue, thick] (3*\k-2.9, -1) -- (3*\k-0.1, 1); 
\draw[blue, thick] (3*\k-2.9, -1) -- (3*\k-0.1, 0.6); 
\draw[blue, thick] (3*\k-2.9, -1) -- (3*\k-0.1, -1); 
\draw[blue, thick] (3*\k-2.9, -1) -- (3*\k-0.1, -1.4); 

\draw[blue, thick] (3*\k-2.9, -1.4) -- (3*\k-0.1, 0.6); 
\draw[blue, thick] (3*\k-2.9, -1.4) -- (3*\k-0.1, 0.2); 
\draw[blue, thick] (3*\k-2.9, -1.4) -- (3*\k-0.1, -1); 
\draw[blue, thick] (3*\k-2.9, -1.4) -- (3*\k-0.1, -1.4);
\draw[blue, thick] (3*\k-2.9, -1.4) -- (3*\k-0.1, -1.8);

\draw[blue, fill=blue] (3*\k,1.4) circle [radius=0.1]; 
\draw[blue, fill=blue] (3*\k,-0.2) circle [radius=0.1];
\draw[blue, fill=blue] (3*\k,-1.8) circle [radius=0.1];
}

\foreach \k in {6,7} {
\draw[blue, thick] (3*\k-2.9, 1.4) -- (3*\k-0.1, 1.0); 
\draw[blue, thick] (3*\k-2.9, 1.4) -- (3*\k-0.1, 1.8); 
\draw[blue, thick] (3*\k-2.9, 1.4) -- (3*\k-0.1, 1.4); 
\draw[blue, thick] (3*\k-2.9, 1.4) -- (3*\k-0.1, -0.2);
\draw[blue, thick] (3*\k-2.9, 1.4) -- (3*\k-0.1, -0.6);

\draw[blue, thick] (3*\k-2.9, -.2) -- (3*\k-0.1, 1.8); 
\draw[blue, thick] (3*\k-2.9, -.2) -- (3*\k-0.1, 1.4); 
\draw[blue, thick] (3*\k-2.9, -.2) -- (3*\k-0.1, 0.2); 
\draw[blue, thick] (3*\k-2.9, -.2) -- (3*\k-0.1, -0.2);
\draw[blue, thick] (3*\k-2.9, -.2) -- (3*\k-0.1, -0.6);
\draw[blue, thick] (3*\k-2.9, -.2) -- (3*\k-0.1, -1.8);
\draw[blue, thick] (3*\k-2.9, -.2) -- (3*\k-0.1, -2.2);

\draw[blue, thick] (3*\k-2.9, -1.8) -- (3*\k-0.1, 0.2); 
\draw[blue, thick] (3*\k-2.9, -1.8) -- (3*\k-0.1, -0.2); 
\draw[blue, thick] (3*\k-2.9, -1.8) -- (3*\k-0.1, -1.4); 
\draw[blue, thick] (3*\k-2.9, -1.8) -- (3*\k-0.1, -1.8); 
\draw[blue, thick] (3*\k-2.9, -1.8) -- (3*\k-0.1, -2.2); 

\draw[blue, fill=blue] (3*\k,-0.6) circle [radius=0.1]; 
\draw[blue, fill=blue] (3*\k,-2.2) circle [radius=0.1]; 
}

\draw[blue, thick] (18.1, -2.2) -- (20.9, -.2); 
\draw[blue, thick] (18.1, -2.2) -- (20.9, -0.6); 
\draw[blue, thick] (18.1, -2.2) -- (20.9, -1.8); 
\draw[blue, thick] (18.1, -2.2) -- (20.9, -2.2); 
\draw[blue, thick] (18.1, -2.2) -- (20.9, -2.6); 

\draw[blue, fill=blue] (21,-2.6) circle [radius=0.1]; 
\end{tikzpicture}  

\caption{Propagation from node 2 in the sparse network with $C_\ell = 1$ and $N_h = 25$. Blue shows the set reached after successive layers; gray shows other admissible connections.} \label{fig:propagation}
\end{center}    
\end{figure}

In our setting, every layer has width $N_h$ and the neurons in every layer are indexed by the same interior-node indices $(i,j) \in \mathcal{V}_h$. 
We let $z^{(0)} = F\in\R^{N_h}$ be an input of a neural network and, for $\ell = 0,1,\dots,L-1$, we  write
\begin{equation}\label{val_layer}
z^{(\ell+1)}_{i,j} = \sigma\left(\sum_{(p,q)\in \mathcal{V}_h} W^{(\ell+1)}_{(i,j),(p,q)} z^{(\ell)}_{p,q}+b^{(\ell+1)}_{{(i,j)}} \right),
\end{equation}
where $\sigma$ is an activation function, $W^{(\ell)}$ is the weight matrix, and $b^{(\ell)}$ is the bias vector for the $\ell$-th layer. The final output is $z^{(L)} \in \mathbb{R}^{N_h}$, which is the prediction for the finite element coefficients. 
In \eqref{val_layer}, the weights are not fully-connected. 
Let $\mathcal{P}^{(1)}_{i,j} = \textrm{supp}(\phi_{i,j})$. For $t \ge 1$, we define the level-$(t+1)$ patch recursively by
\begin{equation}\label{level_patch}
\mathcal{P}^{(t+1)}_{i,j} := \bigcup_{(p,q) \in \mathcal{V}_h}\{\text{supp}(\phi_{p,q}) \ : \  |\textrm{supp}(\phi_{p,q}) \cap \mathcal{P}^{(t)}_{i,j}| \neq 0\}.
\end{equation}
We then define the level-$t$ neighborhood by
$$\mathcal{V}_{t}(i,j) = \{(p,q) \in \mathcal{V}_h : |\textrm{supp}(\phi_{p,q}) \cap \mathcal{P}^{(t)}_{i,j}| \neq 0 \}.$$
Accordingly, for interior nodes $(i,j), (p,q) \in \mathcal{V}_h$, we allow a nonzero weight as follows:
\begin{equation}\label{sparse_weight}
    W^{(k)}_{(i,j),(p,q)} \neq 0 \quad \text{is allowed only if}\ (p,q) \in \mathcal{V}_{C_\ell}(i,j),
\end{equation}
for some given connectivity level $C_\ell \in \mathbb{N}$, and we set the weight to zero otherwise to enforce sparsity.
For a fixed $C_\ell$ and a stack of $L$ such layers, the dependence of $z^{(L)}_{i,j}$ on the input $z^{(0)}$ is supported on $\mathcal{V}_{L C_\ell}(i,j)$.
Hence, the effective receptive field grows linearly with the depth $L$ (see, e.g., Figure~\ref{fig:propagation}).
Details on the effect of connectivity are provided in the appendix.
Remarkably, accurate approximations are achieved even when the effective receptive field does not fully cover the entire domain.

\section{Theoretical Analysis}\label{sec:theory}
In this section, we provide the theoretical background for the proposed sparse FEONet approach. As discussed earlier, our method significantly reduces the number of parameters compared to the original FEONet. While this is clearly advantageous from a computational point of view, it raises a natural theoretical question on the approximation capability of the resulting sparse architecture. In addition, we need to compare our finite-element-guided sparsification strategy with more naive approaches that impose sparsity in a random or purely heuristic manner. In particular, we discuss in what sense the finite-element-guided design leads to theoretical advantages, and how these advantages are reflected in practice. Moreover, we investigate, from a theoretical perspective, whether the proposed sparsity pattern not only improves computational efficiency but also leads to highly stable and efficient training. The goal of this section is to address these theoretical questions for sparse FEONet and to establish rigorous guarantees for its approximation and stability properties. We conclude the section by presenting brief numerical experiments that support the theoretical findings.

\subsection{Universal approximation theorem}
This section is devoted to providing the theoretical justification for the proposed sparse neural networks. In particular, we address the question of whether the proposed sparse neural network constitutes a suitable approximation class. While the universal approximation property of dense neural networks is well established (see, e.g., \cite{UA_1, UA_2, UA_3}), our approach relies on a sparse architecture in which connectivity is restricted to nodes associated with neighboring elements. Thus, it is essential to demonstrate that neural networks with such a structure can still approximate the target function effectively. Accordingly, in this section, we shall first establish a universal approximation theorem for the proposed sparse network. 

Note that for our proposed method, the number of nodes in each layer of our neural network coincides with the number of degrees of freedom $N_h\in\mathbb{N}$. However, in such a setting, it is well known in the literature that even dense neural networks cannot, in general, be guaranteed to exhibit the universal approximation property. In fact, according to the results known to date (e.g., \cite{uat_r_1, uat_r_2, uat_r_3}), a class of neural networks attains the universal approximation property only when the number of nodes in its hidden layer is taken to be sufficiently large. Consequently, one has even less justification for expecting any universal approximation capability from a sparse network defined in the previous section.
Note, however, that the mapping we seek to approximate by means of a neural network does not belong to an arbitrary class of continuous functions. Rather, it is the mapping that assigns, to the parameters defining a given PDE problem, the corresponding coefficients of the associated finite element discretization. As described in \eqref{LAS}, the finite element coefficient can be characterized by a linear algebraic system. Thus, in our framework, the essential task is to represent the linear map $x\mapsto A^{-1}x$ where $A$ denotes the finite element matrix, by a sparse neural network. As will be verified in the subsequent proof, this structural feature aligns with our setting in a particularly precise and favorable manner, and the desired conclusion indeed follows.

To simplify the theoretical setting of the network connectivity, we first describe our proposed method within a graph-theoretic framework. Let us begin with the following definitions.

\begin{definition}[Simple undirected graph]
An \emph{undirected graph} is a pair $G = (V,E)$, where
\begin{itemize}
  \item $V$ is a finite set, whose elements are called \emph{vertices}, and
  \item $E$ is a set of unordered pairs of distinct vertices, i.e.,
  $E \subseteq \big\{ \{u,v\} : u,v \in V,\ u\neq v \big\}$.
\end{itemize}
The elements of $E$ are called \emph{edges}. Moreover, if we do not allow loops (edges of the form $\{v,v\}$) or multiple edges between the same pair of vertices, we call the graph \emph{simple}.
\end{definition}

\begin{definition}[Adjacency]
Let $G=(V,E)$ be a simple undirected graph. Two distinct vertices $u,v\in V$ are said to be \emph{adjacent} if $\{u,v\}\in E$. In this case, we also say that $u$ and $v$ are \emph{joined by an edge}, or that there is an edge between $u$ and $v$.
\end{definition}

\begin{definition}[Path and graph distance]
Let $G=(V,E)$ be a simple undirected graph. For given vertices $v_0$, $v_\ell$, a \emph{path} in $G$ of length $\ell \ge 1$ is a finite sequence of vertices $\{v_0, v_1, \dots, v_\ell\}$ such that $\{v_{m-1}, v_m\} \in E$ for each $m = 1,\dots,\ell$.
The graph distance $d_G(u,v)$ is the length (number of edges) of a shortest path between $u$ and $v$ in $G$.
Each edge has a unit cost, so the distance counts edges.
If no path exists, we set $d_G(u,v) = \infty$.
\end{definition}

\begin{definition}[Connected vertices and connected graph]
Let $G=(V,E)$ be a simple undirected graph. Two vertices $u,v\in V$ are said to be \emph{connected} if either $u=v$ or there exists a path in $G$ from $u$ to $v$. The graph $G$ is called \emph{connected} if every pair of vertices in $V$ is connected.
\end{definition}

\begin{definition}[Connected component]
Let $G=(V,E)$ be a simple undirected graph. A nonempty subset $C \subseteq V$ is called a \emph{connected component} of $G$ if
\begin{itemize}
  \item for any $u,v\in C$, there exists a path in $G$ from $u$ to $v$ (so the induced subgraph on $C$ is connected);
  \item $C$ is maximal with respect to this property: if $C\subseteq C'\subseteq V$ and the induced subgraph on $C'$ is connected, then $C'=C$.
\end{itemize}
The connected components of $G$ form a partition of $V$. We say that $G$ is \emph{disconnected} if it has at least two distinct connected components.
\end{definition}

For a given simple undirected graph $G$, we define a sparse matrix associated with the connectivity of $G$. Here, $\mathrm{GL}_N(\R)$ denotes the group of invertible matrices over $\R$ with respect to the matrix multiplication.

\begin{definition}
Let $G = (V,E)$ be the simple undirected graph with $|V|=N$. A matrix $W = (w_{ij}) \in M_{N}(\R)$ is called \emph{$G$-sparse} if the following hold:
\begin{itemize}
  \item For each $i$, the diagonal entry $w_{ii}$ can be nonzero.
  \item For $i\neq j$, the entries $w_{ij}$ and $w_{ji}$ can be nonzero only if $\{i,j\}\in E$.
\end{itemize}
\end{definition}

\begin{definition} Let $G=(V,E)$ be a simple undirected graph with $|V|=N$. We define $H_{G}(\R)$ to be the subgroup of $\mathrm{GL}_{N}(\R)$ generated by all invertible $G$-sparse matrices. That is,
\[
  H_{G}(\R) := \big\langle M \in \mathrm{GL}_{N}(\R) : M \text{ is $G$-sparse} \big\rangle
  \;\le\; \mathrm{GL}_{N}(\R).
\]
\end{definition}
The first key result in our analysis is the following.

\begin{theorem}\label{thm:main_1}
Let $G$ be a simple undirected graph with $|V|=N$. Then
\[
  H_G(\R) = \mathrm{GL}_N(\R)\text{ if and only if $G$ is connected.}
\]
\end{theorem}
To show this equivalence, let us first prove the implication in one direction, which is encapsulated in the following proposition.

\begin{proposition}\label{prop:disconnected}
If the graph $G=(V,E)$ with $|V|=N$ is disconnected, then $H_G(\R)$ is a proper subgroup of $\mathrm{GL}_N(\R)$, in other words, $H_G(\R) \neq \mathrm{GL}_N(\R)$.
\end{proposition}

\begin{proof}
Suppose that $G=(V,E)$ is disconnected. Then $V$ decomposes into a disjoint union of connected components $V = V_1 \,\cup\, V_2 \,\cup\, \cdots \,\cup\, V_k$, with $k\ge 2$. With a suitable reordering of vertices in $V$, it is straightforward to verify that any $G$-sparse matrix $W$ has the form
\[
  W = \begin{pmatrix}
    W_1 & 0   & \cdots & 0\\
    0   & W_2 & \cdots & 0\\
    \vdots & \vdots & \ddots & \vdots\\
    0   & 0   & \cdots & W_k
  \end{pmatrix},
\]
i.e.\ $W$ is block diagonal, with blocks $W_r$ of size $|V_r|\times |V_r|$ with $1\le r\le k$. Note that the product of block diagonal matrices with this block structure is again block diagonal with the same block structure. Hence, every element of $H_G(\R)$ is of the same form. However, obviously, there exist matrices in $\mathrm{GL}_N(\R)$ that cannot be written in this form. This completes the proof.
\end{proof}

In order to prove the converse, we first recall some standard notation for elementary matrices.
\begin{definition}[transvection matrices]
For $1 \leq i,j \leq N$ with $i\neq j$ and for $t\in \R$, we denote by $E_{ij}(t)$ the matrix
\[
  E_{ij}(t) := I_N + t\,e_{ij},
\]
where $e_{ij}$ is the matrix having a $1$ in position $(i,j)$ and zeros elsewhere, and $I_N$ is the $N\times N$ identity matrix.
\end{definition}

In classical linear algebra, the following fact is well-known, which describes a generating set for $\mathrm{GL}_N(\R)$ in terms of transvections and diagonal matrices. More precisely, the general linear group $\mathrm{GL}_N(\R)$ can be generated by all transvection matrices $E_{ij}(t)$ with $i\neq j$, $t\in \R$ and all invertible diagonal matrices ${\rm{diag}}(\lambda_1,\dots,\lambda_N)$.
 Therefore, to complete the proof of Theorem \ref{thm:main_1}, it remains to show that $H_G(\R)$ contains all transvection matrices and invertible diagonal matrices provided that $G$ is connected. It is obvious that any invertible diagonal matrices ${\rm{diag}}(\lambda_1,\dots,\lambda_n)$ are $G$-sparse, and hence contained in $H_G(\R)$. Moreover, if $\{i,j\}\in E$, it is easy to see directly from the definition that $E_{ij}(t)\in H_G(\R)$ for all $t\in\R$. Therefore, what remains to prove is that $E_{ij}(t)\in H_G(\R)$ for all $t\in\R$ even if $\{i,j\}\notin E$. To do this, let us introduce a commutator identity: for distinct indices $i, j, k$, and $s,t\in\R$, there holds
 \begin{equation}\label{comm_id}
     E_{ij}(st)=E_{ik}(s)\, E_{kj}(t)\, E_{ik}(-s)\, E_{kj}(-t),
 \end{equation}
 which follows immediately by directly computing the matrix multiplications using the identity $e_{pq} e_{rs} = \delta_{qr} e_{ps}$. We now use the connectivity of $G$ to show that $H_G(\R)$ contains every transvection matrices $E_{ij}(t)$ even if $\{i,j\}\notin E$.

\begin{proposition}\label{prop:all-transvections}
Assume that $G=(V,E)$ with $|V|=N$ is connected. Then for every pair of distinct indices $i,j\in \{1,\dots,N\}$ and every $t\in \R$, the transvection $E_{ij}(t)$ belongs to $H_G(\R)$.
\end{proposition}

\begin{proof}
We first fix distinct indices $i,j\in \{1,\dots,N\}$ and consider the vertices corresponding to these indices. Since $G$ is connected, there exists a path $i = v_0,\, v_1,\, \dots,\, v_\ell = j$
in $G$ with $\ell \ge 1$, where each $\{v_{m-1},v_m\}$ is an edge of $G$. We shall prove the claim by induction on $\ell$ that $E_{v_0,v_\ell}(t) \in H_G(\R)$ for all $t\in \R$. For the case of $\ell=1$, we see that $\{v_0,v_1\} = \{i,j\}$ is an edge of $G$, and hence, $E_{ij}(t)\in H_G(\R)$ for all $t\in \R$. Next, suppose that the claim holds for all paths of length at most $\ell\ge 1$. Consider the path $v_0, v_1, \dots, v_\ell, v_{\ell+1}$ of length $\ell+1$. By the induction hypothesis, $E_{v_0,v_\ell}(s)\in H_G(\R)$ for all $s\in \R$. Moreover, the last edge $\{v_\ell, v_{\ell+1}\}$ lies in $E$, and thus $E_{v_\ell,v_{\ell+1}}(t)\in H_G(\R)$ for all $t\in \R$. We now apply the commutator identity \eqref{comm_id} with $i=v_0$, $k=v_\ell$, $j=v_{\ell+1}$, and $s=1$, which leads us to obtain
\[
  E_{v_0,v_{\ell+1}}(t)=E_{v_0,v_\ell}(1) E_{v_\ell,v_{\ell+1}}(t)E_{v_0,v_\ell}(-1) E_{v_\ell,v_{\ell+1}}(-t).
\]
Note that the right-hand side is a product of four matrices lying in $H_G(\R)$. Therefore we can conclude that $E_{v_0,v_{\ell+1}}(t)$ lies in $H_G(\R)$ for $t\in\R$.
\end{proof}

Let us now apply the above theory to our sparse neural network setting to prove the universal approximation theorem. For simplicity, we present the argument in the two-dimensional setting described above. It is, however, straightforward to verify that our approach extends to arbitrary spatial dimensions. Note that the weight matrix $W^{(\ell)}$ in the sparse layer \eqref{val_layer} can be reformulated using the above notation. More precisely, we shall define a simple undirected graph $G_h = (\mathcal{V}_h,E_h)$. For the set of edges $E_h$, we set the vertices $u=(i,j)$, $v=(p,q)\in\mathcal{V}_h$ to be adjacent if $v=(p,q)\in\mathcal{V}_{C_\ell}(i,j)$ for a prescribed constant $C_\ell>0$. The important first step is the following lemma. 
It follows immediately because, for fixed $C_\ell$ and mesh size $h > 0$, the entire domain $(0,1)^2$ can be covered by finitely many \emph{graph balls}
$$\mathcal{B}_G(u; C_\ell) := \{v \in \mathcal{V}_h \ : \ d_G(u, v) \le C_\ell\}.$$
By the patch construction \eqref{level_patch}, $\mathcal{B}_G(u; C_\ell)$ coincides with $\mathcal{V}_{C_\ell}(u)$, which in turn corresponds to $\mathcal{P}^{(C_\ell)}_u$.

\begin{lemma} \label{receptive_field}
    The graph $G_h$ is connected. In other words, for any two interior nodes $(i,j), (p,q) \in \mathcal{V}_h$, there exists a path of nodes $(i,j) = (i_0, j_0), (i_1, j_1), \cdots, (i_m, j_m) = (p,q)$
    such that $d_G(\eta_{i_{k+1}, j_{k+1}},\ \eta_{i_k, j_k}) \leq C_{\ell}$ for all $k=0,1,\cdots,m-1$.
\end{lemma}

\begin{remark}[High order approximation]
The graph construction can be naturally extended for higher-order approximations.
We define the graph $G = (V, E)$ whose vertex set $V$ consists of all degrees of freedom (vertex, edge, face, and interior nodes).
Two nodes $u,v \in V$ are adjacent, i.e., $\{u, v\} \in E$, if the supports of their basis functions overlap on a set of positive measure.
Each edge has unit cost; the neighborhood $\mathcal{V}_{C_\ell}(u) = \{v\ : \ d_G(u,v) \le C_\ell\}$ and the graph balls are defined with respect to $d_G$.
For the piecewise linear approximation, this reduces to the vertex-adjacency graph used above.
\end{remark}

An important observation is that our sparse weight matrix $W^{(\ell)}$ in \eqref{val_layer} can be characterized by the graph $G_h$. To be more specific, if we construct the network sparsely according to the procedure proposed in this Section \ref{sec:method}, then our weight matrices $W^{(\ell)}$ are $G_h$-sparse for $\ell=1,2,\ldots, L$. Now, we are ready to prove the universal approximation property of our sparse network. Since the graph $G_h$ corresponding to the sparse layer defined above is connected and the finite element matrix $A$ is invertible, by Theorem \ref{thm:main_1}, we see that $A^{-1}$ can be represented as a product of $G_h$-sparse matrices. Based on this fact, we shall prove that our target function $x\mapsto A^{-1}x$ can be represented as a ReLU network with sparse layers.
In other words, each $G_h$-sparse matrix can be realized as a single affine layer of a neural network with the prescribed sparsity pattern, which allows us to translate the above matrix factorization into a neural network representation.

\begin{theorem}[Universal approximation for ReLU sparse networks]
\label{thm:G-ReLU-realization}
Let $M\in\mathrm{GL}_{N_h}(\mathbb{R})$ be an invertible matrix, and $K\subset\mathbb{R}^{N_h}$ be a nonempty compact set. Then there exists a ReLU neural network $\mathcal{N} : \mathbb{R}^{N_h} \to \mathbb{R}^{N_h}$ whose weight matrices are all $G_h$-sparse and satisfies
\[
  \mathcal{N}(x) = M x \quad \text{for all } x\in K.
\]
\end{theorem}

The idea of the proof is to implement each factor of $M$ by a small ReLU subnetwork that acts as the identity on the relevant compact subset of intermediate representations. The key observation is that ReLU coincides with the identity on the positive half-line.

\begin{proof}
Since $G_h$ is connected, Theorem \ref{thm:main_1} yields $H_{G_h}(\mathbb{R}) = {\rm{GL}}_{N_h}(\mathbb{R})$. Hence for any $M\in {\rm{GL}}_{N_h}(\mathbb{R})$ there exist $m\in\mathbb{N}$ and $G_h$-sparse matrices $M_1,\dots,M_m$ such that
\[
M = M_m M_{m-1}\cdots M_1 .
\]
We shall construct a depth-$(m+1)$ network by setting $z^{(0)}=x$, $z^{(\ell)}(x) := \sigma\!\bigl(M_\ell z^{(\ell-1)}(x) + b_\ell\bigr)$ for $\ell=1,\dots,m$, and $\mathcal{N}(x) := z^{(m)}(x) + b_{m+1}$, where $\sigma$ denotes the ReLU activation function applied componentwise. We choose the biases so that all pre-activations are strictly positive on $K$, forcing $\sigma$ to act as the identity on that regime. For $\ell=1$, for each coordinate $i$ define
\[
\alpha_{1,i} := \min_{x\in K} (M_1 x)_i,
\]
which exists by compactness of $K$ and continuity of a linear map. We shall choose $b_1\in\mathbb{R}^{N_h}$ such that $(b_1)_i > -\alpha_{1,i}$ for all $i$. Then we see that
$z^{(1)}(x)=M_1x+b_1$. Next, note that $z^{(\ell-1)}(K)$ is compact for all $\ell$. For each $i$, we set
\[
\alpha_{\ell,i} := \min_{x\in K} \bigl(M_\ell z^{(\ell-1)}(x)\bigr)_i,
\]
and choose $b_\ell$ with $(b_\ell)_i>-\alpha_{\ell,i}$ for all $i$. Then we have
$z^{(\ell)}(x)=M_\ell z^{(\ell-1)}(x)+b_\ell$. Therefore, for all $x\in K$, we obtain
\[
z^{(m)}(x)=M_m\cdots M_1 x \;+\; \sum_{k=1}^{m} M_m\cdots M_{k+1} b_k.
\]
Finally, by setting $b_{m+1}:=-\sum_{k=1}^{m} M_m\cdots M_{k+1} b_k$, we obtain the desired result.
\end{proof}

Theorem \ref{thm:G-ReLU-realization} means that any given invertible linear mapping can be represented exactly by a ReLU sparse network. A natural subsequent question is whether an analogous property holds for more general activation functions. In this case, as in the classical universal approximation theorem, we can approximate any given invertible linear mapping to arbitrary accuracy, which is encapsulated in the following theorem.

\begin{theorem}[Universal approximation for sparse networks with general activation]\label{thm:gen_act}
    Let $\sigma:\mathbb{R}\to\mathbb{R}$ be an activation function such that there exist
$t_0\in\mathbb{R}$ and an open interval $U$ containing $t_0$ such that $\sigma\in C^1(U)$ and $\sigma'(t_0)\neq 0$.
Then for every $M\in {\rm{GL}}_{N_h}(\mathbb{R})$, every nonempty compact set $K\subset\mathbb{R}^{N_h}$, and every $\varepsilon>0$,
there exists a finite-depth $\sigma$-network $\mathcal{N}:\mathbb{R}^{N_h}\to\mathbb{R}^{N_h}$ whose weight matrices are
$G_h$-sparse in every layer such that
\[
\sup_{x\in K}\|\mathcal{N}(x)-Mx\|_\infty < \varepsilon.
\]
\end{theorem}
\begin{remark}
    The assumptions for the activation in Theorem \ref{thm:gen_act} are satisfied by most commonly used activation functions, including tanh, sigmoid, softplus, GELU, and Swish.
\end{remark}
\begin{proof}
    As before, since $G_h$ is connected, there exist $m\in\mathbb{N}$ and $G_h$-sparse matrices $M_1,\dots,M_m$ such that
\[
M = M_m M_{m-1}\cdots M_1 .
\]
For $\delta>0$, let us define 
\[
\phi_\delta(u) := \frac{\sigma(t_0+\delta u)-\sigma(t_0)}{\delta\,\sigma'(t_0)}.
\]
Since $\sigma\in C^1(U)$, from the first-order Taylor expansion, we see that for every $R>0$ and $\eta>0$, there exists sufficiently small $\delta>0$ such that 
\[\sup_{|u|\le R}|\phi_\delta(u)-u|<\eta.\]
If we define the componentwise extension $\Phi_\delta:\mathbb{R}^{N_h}\to\mathbb{R}^{N_h}$ by $\Phi_\delta(z)_i:=\phi_\delta(z_i)$, then we also have that for every $R>0$ and $\eta>0$, there exists $\delta>0$ such that
\begin{equation}\label{eq:Phi-close}
\sup_{\|z\|_\infty\le R}\|\Phi_\delta(z)-z\|_\infty \le \eta.
\end{equation}

Now, for each $\ell=1,\dots,m$ and $\delta>0$ define the two-layer $\sigma$-block
\begin{equation}\label{eq:Tblock}
\mathcal{T}_{\ell,\delta}(z)
:= \frac{1}{\delta\sigma'(t_0)}\Bigl(\sigma(\delta M_\ell z + t_0\mathbf{1})-\sigma(t_0)\mathbf{1}\Bigr),
\end{equation}
where $\mathbf{1}\in\mathbb{R}^{N_h}$ is the all-ones vector and $\sigma$ is applied componentwise. Note that the weight matrices for the layers $\delta M_\ell$ and $\frac{1}{\delta\sigma'(t_0)}I$ are both $G_h$-sparse.
Moreover, by definition, we see that $\mathcal{T}_{\ell,\delta}(z)=\Phi_\delta(M_\ell z)$. Therefore, for any $R>0$ and $\eta>0$, choosing small $\delta$ so that \eqref{eq:Phi-close} holds with radius
$\|M_\ell\|R$ yields
\begin{equation}\label{eq:block-approx}
\sup_{\|z\|_\infty\le R}\|\mathcal{T}_{\ell,\delta}(z)-M_\ell z\|_\infty
=\sup_{\|z\|_\infty\le R}\|\Phi_\delta(M_\ell z)-M_\ell z\|_\infty
\le \eta.
\end{equation}
Let us define the radii for each layer by
\[
R_0 := \sup_{x\in K}\|x\|_\infty <\infty,
\quad
R_\ell := \|M_\ell\|R_{\ell-1}+1 \quad (\ell=1,\dots,m).
\]
Let us also define the amplification constants
\[
A_\ell := \prod_{j=\ell+1}^{m}\|M_j\| \quad(\text{with }A_m=1),\quad
A_{\max}:=\max_{1\le \ell\le m}A_\ell.
\]
For given $\varepsilon>0$, we shall choose
\[
\eta := \min\Bigl\{1,\ \frac{\varepsilon}{m\,A_{\max}}\Bigr\}.
\]
For each $\ell$, from \eqref{eq:block-approx} with radius $R_{\ell-1}$ and $\eta$ to pick $\delta_\ell>0$ such that
\begin{equation}\label{eq:each-layer}
\sup_{\|z\|_\infty\le R_{\ell-1}}\|\mathcal{T}_{\ell,\delta_\ell}(z)-M_\ell z\|_\infty \le \eta.
\end{equation}
We now define the network
\[
\mathcal{N} := \mathcal{T}_{m,\delta_m}\circ \cdots \circ \mathcal{T}_{1,\delta_1}.
\]
Note that each block $\mathcal{T}_{\ell,\delta_\ell}$ can be written as
$\mathcal{T}_{\ell,\delta_\ell}=A^{(2)}_\ell\circ \sigma \circ A^{(1)}_\ell$ with affine maps
$A^{(1)}_\ell(z)=\delta_\ell M_\ell z+t_0\mathbf{1}$ and
$A^{(2)}_\ell(y)=\frac{1}{\delta_\ell\sigma'(t_0)}y-\frac{\sigma(t_0)}{\delta_\ell\sigma'(t_0)}\mathbf{1}$.
Hence the composition $\mathcal{N}:=\mathcal{T}_{m,\delta_m}\circ\cdots\circ \mathcal{T}_{1,\delta_1}$
is a standard feedforward $\sigma$-network obtained by merging consecutive affine maps $A^{(1)}_{\ell+1}\circ A^{(2)}_\ell$ into a single affine map, where the $G_h$-sparsity is preserved.

Now, let us fix $x\in K$ and define the exact and approximate trajectories, respectively by
\[
z^{(0)}:=x,\quad z^{(\ell)}:=M_\ell z^{(\ell-1)};
\qquad
\tilde z^{(0)}:=x,\quad \tilde z^{(\ell)}:=\mathcal{T}_{\ell,\delta_\ell}(\tilde z^{(\ell-1)}).
\]
By definition, for any $\ell=1,\ldots,m$, we see that
\[
\|\tilde z^{(\ell)}\|_\infty
\le \|M_\ell \tilde z^{(\ell-1)}\|_\infty + \eta
\le \|M_\ell\|R_{\ell-1}+1 = R_\ell.
\]
Thus \eqref{eq:each-layer} applies at every stage along the approximate trajectory. Next, for the error $e_\ell:=\|\tilde z^{(\ell)}-z^{(\ell)}\|_\infty$,
\begin{align*}
e_\ell
&= \bigl\|\mathcal{T}_{\ell,\delta_\ell}(\tilde z^{(\ell-1)}) - M_\ell z^{(\ell-1)}\bigr\|_\infty \\
&\le \bigl\|\mathcal{T}_{\ell,\delta_\ell}(\tilde z^{(\ell-1)}) - M_\ell \tilde z^{(\ell-1)}\bigr\|_\infty
     + \|M_\ell(\tilde z^{(\ell-1)}-z^{(\ell-1)})\|_\infty \\
&\le \eta + \|M_\ell\| e_{\ell-1},
\end{align*}
for $\ell=1,\ldots,m$ where we used \eqref{eq:each-layer} with $\|\tilde z^{(\ell-1)}\|_\infty\le R_{\ell-1}$. 
Iterating this recursion yields
\[
\|\mathcal{N}(x)-Mx\|_\infty=e_m \le \sum_{\ell=1}^m A_\ell\,\eta \le m\,A_{\max}\,\eta \le \varepsilon,
\]
which completes the proof.
\end{proof}

\begin{figure}[t]
    \centering
      
    \begin{subfigure}{0.23\textwidth}
        \centering
        \includegraphics[width=\linewidth]{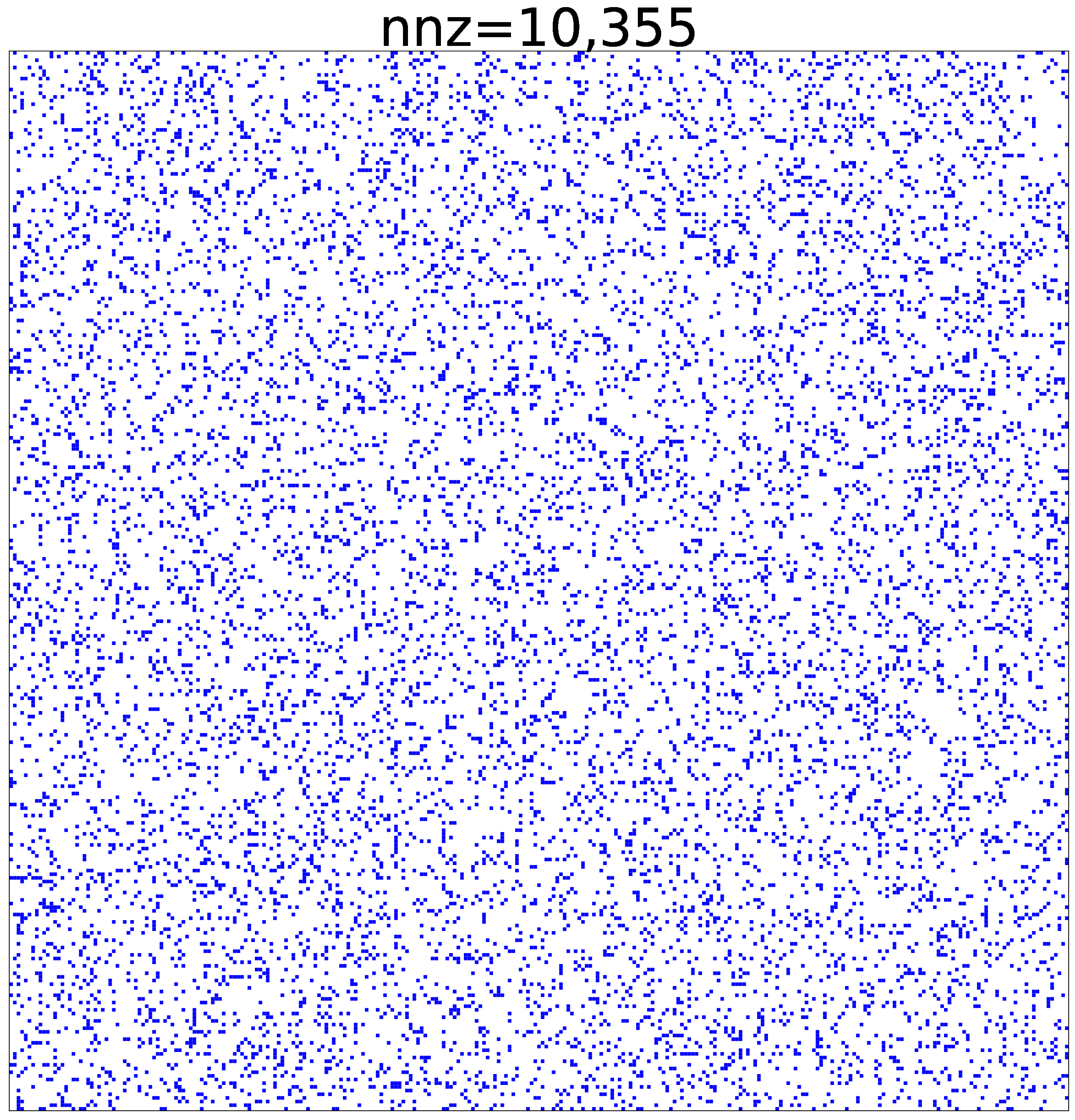}
        \caption{}
        \label{fig:16_random}
    \end{subfigure}
    \begin{subfigure}{0.23\textwidth}
        \centering
        \includegraphics[width=\linewidth]{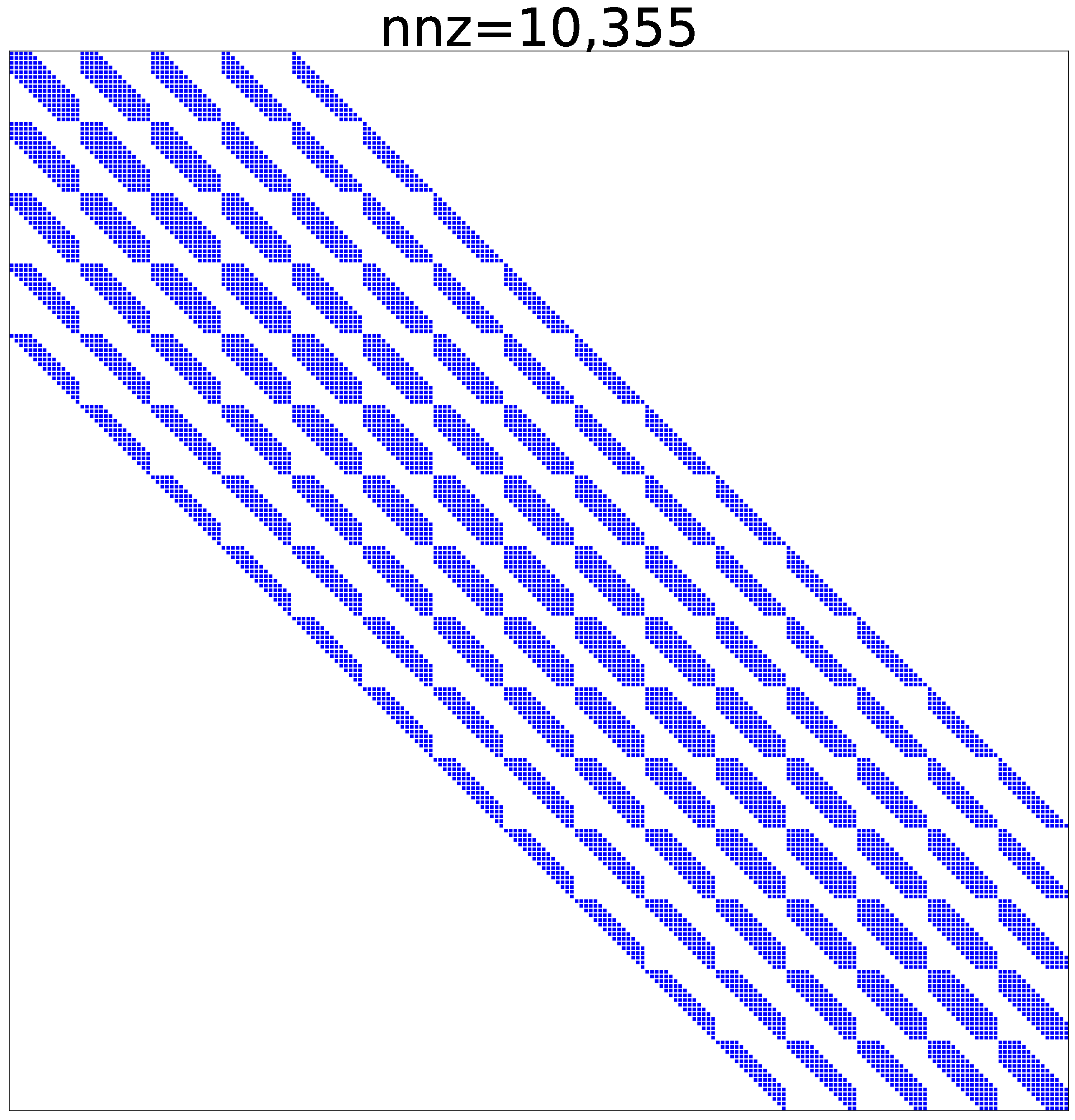}
        \caption{}
        \label{fig:16_local}
    \end{subfigure}
    \begin{subfigure}{0.23\textwidth}
        \centering
        \includegraphics[width=\linewidth]{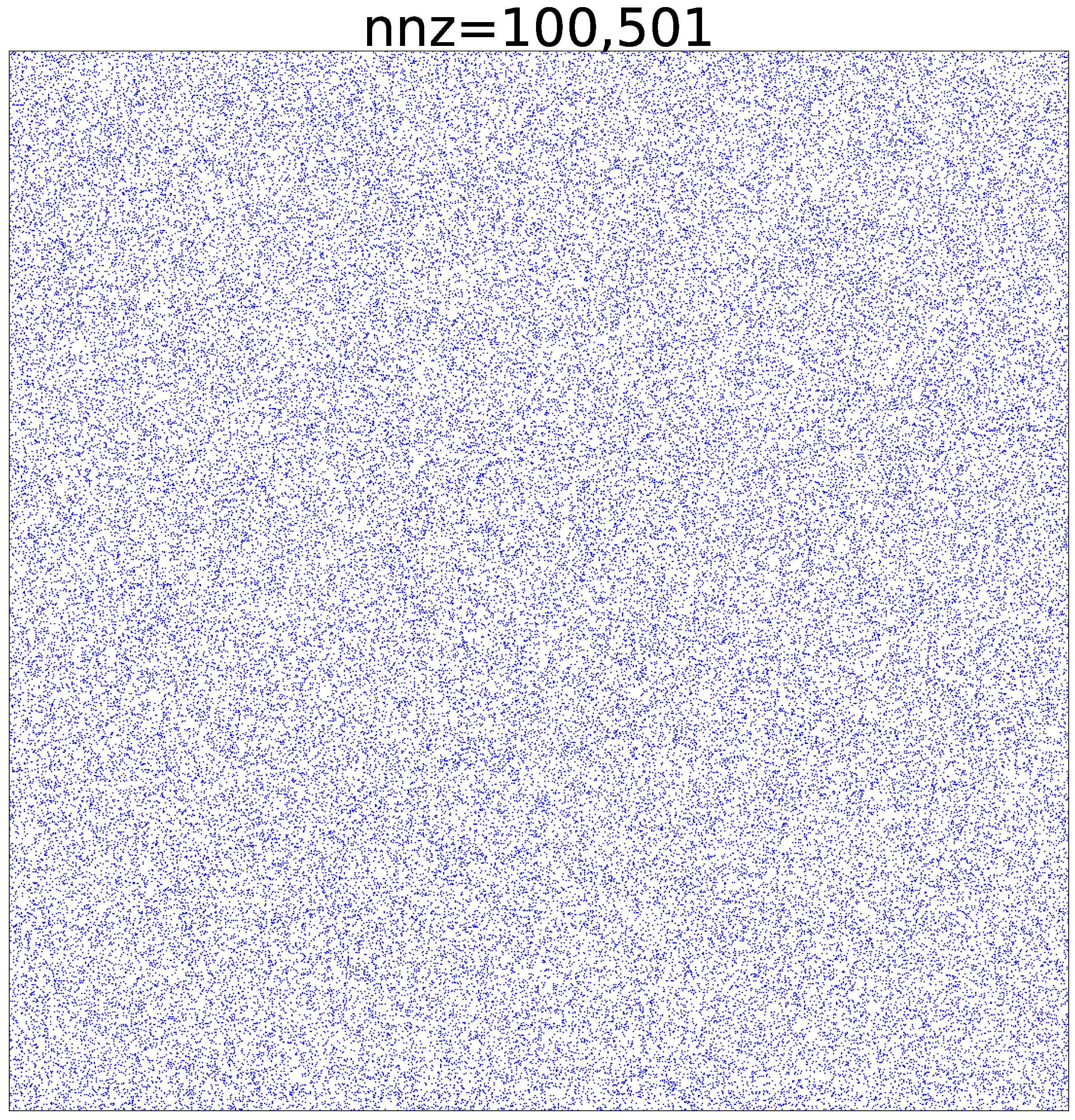}
        \caption{}
        \label{fig:32_random}
    \end{subfigure}
    \begin{subfigure}{0.23\textwidth}
        \centering
        \includegraphics[width=\linewidth]{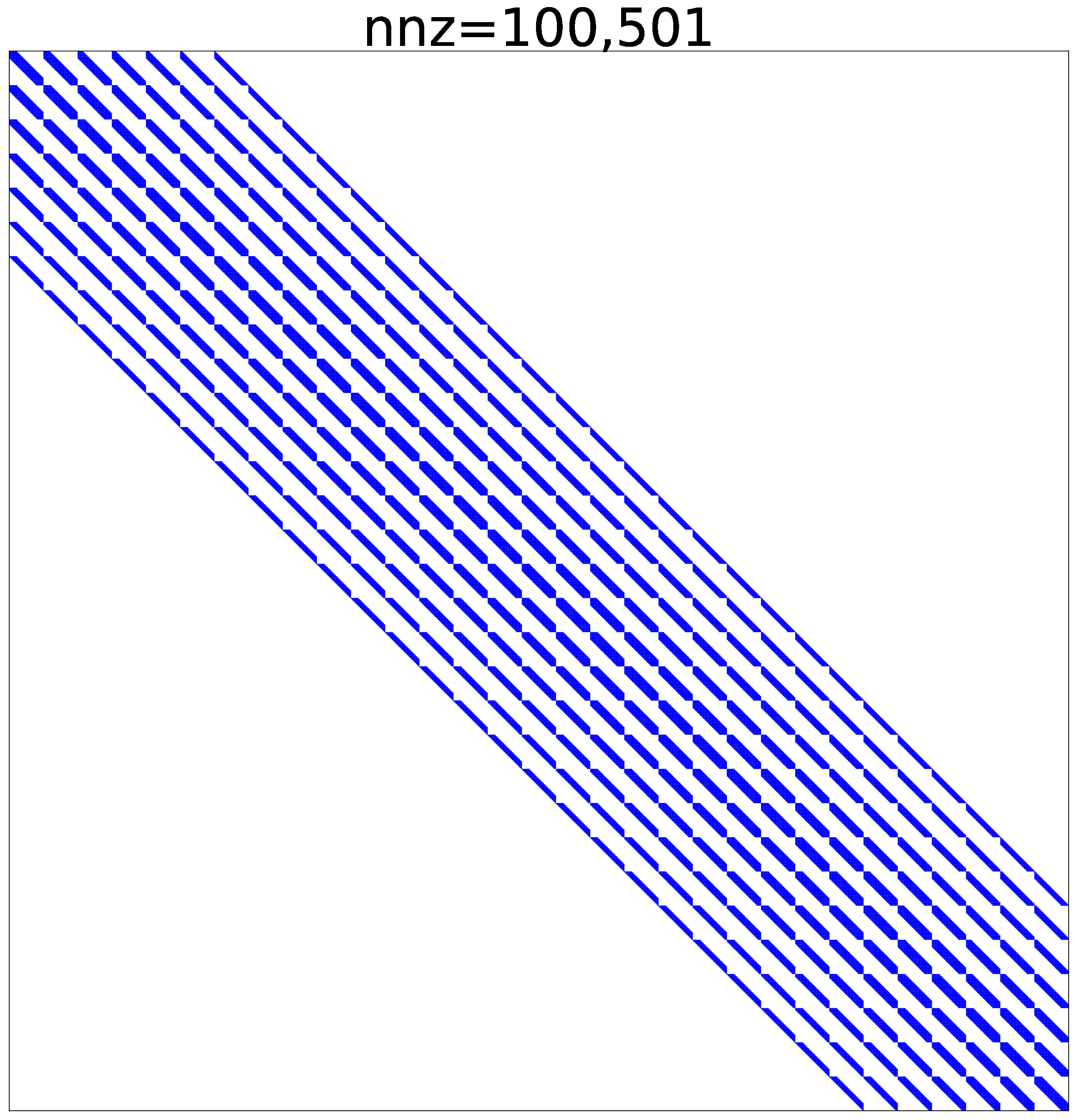}
        \caption{}
        \label{fig:32_local}
    \end{subfigure}
    \caption{
    Sparse weight matrix patterns: (a) random ($N_h=225$), (b) our method ($N_h = 225$, $C_\ell = 4$), (c) random ($N_h = 961$), and (d) our method ($N_h = 961$, $C_\ell = 6$). The plot of random connectivity patterns is selected from one of the 10 random seeds.
    }
    \label{fig:connectivity_test}
\end{figure}

As we can see from Theorem \ref{thm:G-ReLU-realization} and Theorem \ref{thm:gen_act}, under a given mesh size $h$ and a fixed connectivity constant $C_{\ell}$, the connectivity of the underlying graph plays a crucial role in guaranteeing the above results. In particular, when the graph is connected so that information from the input can propagate to the final layer, we are able to establish a universal approximation property. From this viewpoint, it is natural to wonder whether one might lose such desirable guarantees if, instead of designing the network as proposed here, one imposes sparsity in an ad hoc manner to construct a sparse neural network. To investigate this, we performed some numerical experiments. For $N_h=225$ and $N_h=961$, we constructed two sparse networks, one using our proposed method and the other using a randomly imposed sparsity pattern, and then carried out the FEONet experiments with both architectures. Figure \ref{fig:connectivity_test} illustrates the structure of the weight matrices used in each case. The blue entries indicate nonzero values, whereas the remaining white entries correspond to zeros, yielding sparse matrices. On a uniform mesh, the proposed sparse weight matrices have a banded structure, which is shown in (b) and (d) of Figure~\ref{fig:connectivity_test}.
For a fair comparison, we construct a random sparse weight matrix with the same number of nonzero entries as the proposed one, and we exclude any case with an all-zero row or an all-zero column (see (a) and (c) of Figure~\ref{fig:connectivity_test}). 

As we can see from Table \ref{tab:comparison_connectivity}, the sparsity pattern introduced by the supports of the FEM basis functions yields substantially higher efficiency than a random sparse pattern. For each $N_h$, we generated 10 different random sparse weight matrices and built ten corresponding networks. These were trained for 10,000 epochs under exactly the same settings as the sparse FEONet. The relative \(\ell_2\) Errors are calculated with the finite element solution at the same $N_h$ as the reference solution. This confirms that the connectivity condition discussed above plays a central role, both in the theory and in practical numerical performance.

\begin{table}[t]
\centering 
\caption{Comparison of connection strategies between random and our FEM-based local connectivity on \(N_h=225\) and \(N_h=961\). Relative \(\ell_2\) Errors are calculated on the test set based on the finite element solution at the same resolution $n$. For the random connectivity, the reported error is the mean of 10 different random seed tests. FEM-based connectivity shows more stable and accurate convergence than random connectivity.}
\label{tab:comparison_connectivity}  
\scalebox{0.8}{
\begin{tabular}{@{\hspace{20pt}} c @{\hspace{25pt}} c @{\hspace{40pt}} c @{\hspace{25pt}} c @{\hspace{25pt}}}
    \toprule
    \textbf{\(N_h\)} &
    \textbf{Connection Strategy} &
    \textbf{Number of Connection} &
    \textbf{Rel. \(\ell_2\) Err.} \\
    \midrule
    
    \multirow{2}{*}{225} & FEM-Based local connection & $C_\ell=4$ & \textbf{0.00067} \\
     & Random connection & - & 0.08217 \\
    \midrule

    \multirow{2}{*}{961} & FEM-Based local connection & $C_\ell=6$& \textbf{0.00058} \\
     & Random connection & - & 0.04473 \\
    \bottomrule
    
    \end{tabular}
}
\end{table}

\subsection{Stability}
In this section, in order to further highlight the efficiency of the proposed method, we present a theoretical study of the network's stability in training and inference.
The analysis relies on layer-wise operator norms and on an activation with Lipschitz continuity, and it yields stability bounds with explicit dependence on depth and resolution.
We compare dense (FC) and sparse connectivity and find that the sparse network preserves locality and yields stronger stability guarantees. We shall also present some brief numerical tests to support the theoretical result.

By the Marchenko--Pastur law \cite{p_law}, if each entry of $W\in\mathbb{R}^{N\times N}$ is independently identically distributed random variables with mean 0 and variance $\sigma^2 < \infty$, i.e., $W_{ij} \sim \mathcal{N}(0,\sigma^2)$ for all $i,j\in\{1,2,\cdots,N\}$, it is known, with high probability, that
\begin{equation}\label{Marchenko_estimate}
    \|W\|_2 \approx 2\sigma \sqrt{N}.
\end{equation}
On the other hand, the key observation is that for a sparse matrix $W \in\mathbb{R}^{N\times N}$ constructed via \eqref{sparse_weight}, we have
\begin{equation}\label{sparse_bound}
    \|W\|_2 \le (\|W\|_1 \|W\|_\infty)^{1/2} \le  \omega \gamma = \mathcal{O}(1),
\end{equation}
where $\gamma =\max_{i,j} |\mathcal{V}_{C_\ell}(i,j)|$ denotes the maximum number of nonzero entries per row.
Note that $\gamma$ depends on the constant $C_\ell$ and the mesh structure, but it is independent of $N$. With this in mind, we compare the stability of the dense (FC) network with that of our proposed sparse network, which is encapsulated in the following theorem.

\begin{theorem} [Stability] \label{stability_thm}
Let $\mathcal{N}_L= \Phi_L \circ \cdots \circ \Phi_1: \mathbb{R}^{N_h} \to \mathbb{R}^{N_h}$ be an $L$-layer neural network with
\[
    \Phi_\ell(x) = \sigma(W^{(\ell)}x + b^{(\ell)}), \quad\ell=1,\ldots,L,
\]
where the activation $\sigma$ is Lipschitz continuous with the Lipschitz constant $L_\sigma$, and the weight matrix $W^{(\ell)}$ is either a dense matrix (FC) or a sparse matrix as suggested in \eqref{sparse_weight} whose entries are drawn i.i.d. from $\mathcal{N}(0,\sigma^2)$. Then, for an input $f\in \mathbb{R}^{N_h}$ and the perturbed input $\hat{f} = f + \delta(f)$, we have
    \begin{equation} \label{stability_estimate}
        \|\mathcal{N}_L(f) - \mathcal{N}_L(\hat{f})\|_2 \le C_S \|\delta(f)\|_2,
    \end{equation}
where $C_S = \mathcal{O}((N_h)^\frac{L}{2})$ for the FC case and $C_S = \mathcal{O}(1)$ for the sparse case.
\end{theorem}
\begin{proof}
For each $\ell=1,\ldots,L$, let us write $C_{W^{(\ell)}} = \|W^{(\ell)}\|_2$. From the Lipshcitz continuity of $\sigma$, we see that
    \begin{align*}
        \|\Phi_\ell(x_1) - \Phi_\ell(x_2)\|_2 &\le L_\sigma  \|W^{(\ell)}(x_1 - x_2)\|_2 \le L_\sigma \|W^{(\ell)}\|_2 \|x_1 - x_2\|_2  \le L_\sigma C_{W^{(\ell)}} \|x_1 - x_2\|_2,
    \end{align*}
    for all $x_1, x_2 \in \mathbb{R}^{N_h}$.
    Therefore, by iteration, it follows that
    \begin{align*}
        \|\mathcal{N}_L(f) - \mathcal{N}_L(\hat{f})\|_2 &= \|\Phi_L \circ \cdots \circ \Phi_1(f) - \Phi_L \circ \cdots \circ \Phi_1(\hat{f})\|_2 \nonumber \\
        &\le L_\sigma C_{W^{(L)}} \|\Phi_{L-1} \circ \cdots \circ \Phi_1(f) - \Phi_{L-1} \circ \cdots \circ \Phi_1(\hat{f})\|_2 \nonumber \\
        &\le \cdots \le (L_\sigma)^L \left(\prod_{\ell=1}^L C_{W^{(\ell)}}\right) \|f - \hat{f}\|_2. 
    \end{align*}
    Note that, from \eqref{Marchenko_estimate} and \eqref{sparse_bound}, we see that $C_{W^{(\ell)}}=\mathcal{O}(N_h^{\frac{1}{2}})$ for the FC case, and $C_{W^{(\ell)}}=\mathcal{O}(1)$ for the sparse case, which completes the proof.
\end{proof}

\begin{figure}[!t]
\begin{center}
    \includegraphics[width=0.49\textwidth]{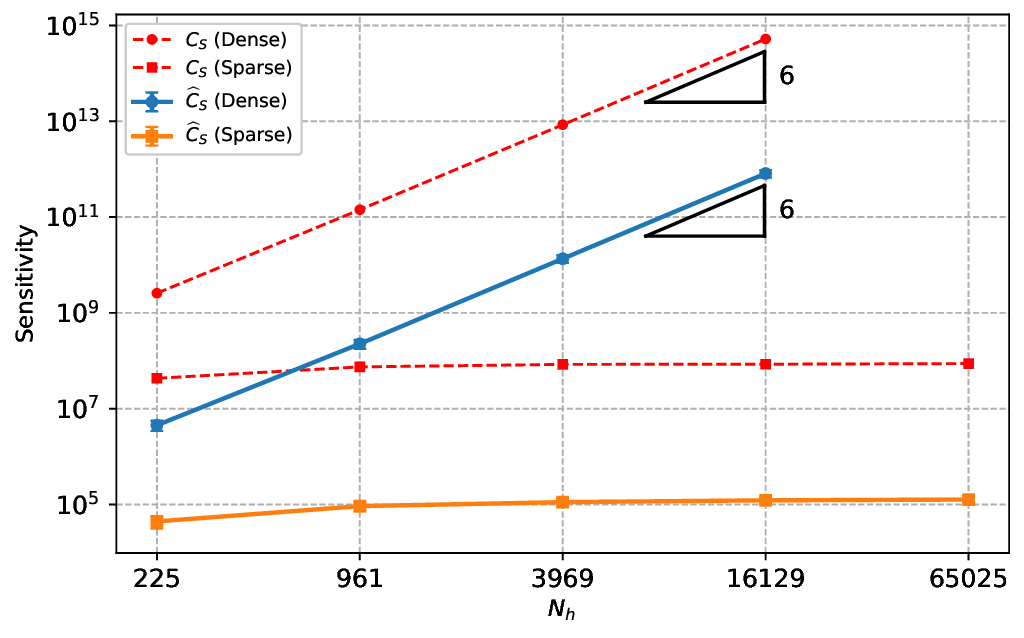}
    \caption{
    Dense vs. sparse stability ($L = 6$) with i.i.d. Gaussian weights. Curves show the empirical sensitivities and the theoretical upper bounds. The dense case at $N_h = 65,025 \ (= 255^2)$ is omitted due to memory limits.
    }
    \label{fig:stability_result}
\end{center}
\end{figure}

\begin{remark}
For a fixed $L$, as presented in Figure~\ref{fig:stability_result}, FC networks may become numerically unstable as $N_h$ increases, while sparse networks remain numerically stable. We note that, since the operator norm of each layer controls signal amplification during both forward and backward propagation, a large global Lipschitz constant $C_S$ in \eqref{stability_estimate} can make the FC network numerically ill-conditioned. For the sparse architecture, the error bound is independent of $h$ since $C_S = \mathcal{O}(1)$. However, it may grow with $N_h$ for the FC architecture because $C_S = \mathcal{O}((N_h)^{\frac{L}{2}})$.
We note that, for commonly used activation functions such as ReLU, tanh, softplus, GELU, and Swish, the Lipschitz constant is uniformly bounded by a moderate constant; 
therefore, it does not affect the asymptotic dependence on $N_h$.
\end{remark}

\begin{table}[!t]
\centering
\caption{Comparison of the stability upper bound and the spectral norm of the weight layer for varying $N_h$. The weights are initialized by Gaussian random sampling and evaluated without training. The norm of input perturbation is set to 1\% of the maximum norm of 3{,}000 input data samples. The results show that the spectral norm of each weight layer in the FC model is increasing when $N_h$ is increasing, whereas that of the sparse model remains nearly constant.}
\label{tab:stability_gaussian_random_weight}

\renewcommand{\arraystretch}{1.2}

\begin{tabular}{c | c | r r r r r r | c}
    \hline  
    \textbf{$N_h$} &
    \textbf{Network} &
    \textbf{$C_{W^{(1)}}$} &
    \textbf{$C_{W^{(2)}}$} &
    \textbf{$C_{W^{(3)}}$} &
    \textbf{$C_{W^{(4)}}$} &
    \textbf{$C_{W^{(5)}}$} &
    \textbf{$C_{W^{(6)}}$} &
    \textbf{$\widehat{C}_S$}\\
    \hline

    \multirow{2}{*}{$15^2$}
    & Dense & 33.954 & 33.331 & 33.971 & 33.505 & 33.761 & 33.417 & $4.52 {\scriptstyle (\pm 1.08)} \times10^{6}$ \\
    & Sparse ($C_\ell =5$) & 17.078 & 17.185 & 16.843 & 16.683 & 17.593 & 16.645 & $4.44 {\scriptstyle (\pm 1.27)}\times10^{4}$\\
    \hline

    \multirow{2}{*}{$31^2$}
    & Dense & 66.012 & 65.708 & 65.677 & 65.629 & 65.395 & 65.935 & $2.26 {\scriptstyle (\pm 0.48)} \times10^{8}$ \\
    & Sparse ${(C_\ell =5)}$ & 19.120 & 18.348 & 18.821 & 18.676 & 18.590 & 18.783 & $9.26{\scriptstyle (\pm 1.99)}\times10^{4}$ \\
    \hline
    
    \multirow{2}{*}{$63^2$}
    & Dense & 130.070 & 129.618 & 129.703 & 129.625 & 129.937 & 129.819 & $1.35 {\scriptstyle (\pm 0.24)} \times10^{10}$ \\
    & Sparse ${(C_\ell =5)}$ & 18.743 & 19.073 & 18.895 & 18.910 & 19.045 & 19.028 & $1.12 {\scriptstyle (\pm 2.06)} \times10^{5}$ \\
    \hline
    
    \multirow{2}{*}{$127^2$}
    & Dense & 257.952 & 257.901 & 257.705 & 257.755 & 257.821 & 257.732 & $8.12 {\scriptstyle (\pm 1.47)} \times10^{11}$ \\
    & Sparse ${(C_\ell =5)}$ & 19.152 & 19.091 & 19.001 & 19.163 & 19.012 & 19.011 & $1.22 {\scriptstyle (\pm 2.14)} \times10^{5}$ \\
    \hline
    
    \multirow{2}{*}{$255^2$}
    & Dense & \multicolumn{1}{c}{--} & \multicolumn{1}{c}{--} & \multicolumn{1}{c}{--} & \multicolumn{1}{c}{--} & \multicolumn{1}{c}{--} & \multicolumn{1}{c|}{--} & -- \\
    & Sparse ${(C_\ell =5)}$ & 19.126 & 19.195 & 19.148 & 19.068 & 19.191 & 19.126 & $1.26 {\scriptstyle (\pm 2.25)} \times10^{5}$ \\
    \hline
    
    \end{tabular}
\end{table}

In order to demonstrate Theorem~\ref{stability_thm}, we test depth $L = 6$ under two connectivities: FC layer and sparse layer that satisfies \eqref{sparse_weight}.
We consider both untrained i.i.d. Gaussian weights and trained weights, under identical architectures and training conditions.
We generated 3,000 input samples and added random noise with a magnitude with $1\%$ of the maximum norm of the input dataset. We then compared the model outputs for the original and perturbed inputs.
We define the empirical sensitivity by
$$\widehat{C}_S(f,\delta) = \frac{\|\mathcal{N}_L(f) - \mathcal{N}_L(\hat{f})\|_2}{\|\delta(f)\|_2}.$$
For each resolution $N_h$, we report the mean (and standard deviation) of $\widehat{C}_S$ over the 3,000 samples and compare it with the theoretical upper bound $C_S$.

By the Marchenko--Pastur estimate \eqref{Marchenko_estimate}, each $C_{W^{(\ell)}}$ for the FC network doubles when $N_h$ is quadrupled. So the right-hand side in \eqref{stability_estimate} increases by a factor $2^{L} = 2^6$.
Figure~\ref{fig:stability_result} and Table~\ref{tab:stability_gaussian_random_weight} confirm that the upper bound $C_S$ grows with slope $\approx 6$ on a log scale, whereas the empirical sensitivity remains strictly below the upper bound for all $N_h$. We note that the FC network could not be instantiated due to memory limitations at $N_h = 255^2$. For the sparse network, $C_\ell = \mathcal{O}(1)$ is independent of $N_h$. 

\section{Numerical Experiments}\label{sec:exp}
In this section, we present the experimental results comparing our proposed sparse network with the FC neural network within the FEONet framework. We tested various PDEs with Dirichlet boundary conditions, from coarse to fine mesh resolutions. 

For each experiment, we utilized 3,000 training samples and 3,000 test samples from randomly generated external forcing terms. As the mesh becomes finer, the number of trainable parameters increases significantly, which leads to memory bottlenecks and computational infeasibility in FC architectures. However, our model overcame these problems, improving efficiency and achieving higher accuracy even in a high-resolution regime. To evaluate the robustness of our sparse model across diverse geometries, we tested the model not only on structured triangular meshes in a square domain, but also on irregular triangular meshes in both square and circular-hole domains generated via FEniCS \cite{fenics}. These tests demonstrate that the proposed model shows strong performance across diverse domains and mesh geometries. Both the FC model and the proposed sparse model were trained under identical settings on the same dataset for a fair comparison. We increased the connectivity step by step to find out the minimum number of connections needed to achieve accuracy comparable to the fully connected baseline model. Each model employed five hidden layers and was optimized using the Adam optimizer. We used the Swish activation function and applied a cosine decay scheduler to gradually reduce the learning rate for better training convergence. The training was performed in an unsupervised manner by minimizing the weak-form residual.

For the performance evaluation of the proposed method, we first measure the relative $L^2$ error and the relative $H^1$ semi-norm error against the FEM solution $u_h^*$ on a fine mesh ($n=1,024$) which can be regarded as a true solution. To be more specific, we measure the following errors:
$$\textrm{Rel. $L^2$ Err.} := \frac{\|u_{\textrm{pred}} -  u_h^*\|_{L^2(\Omega)}}{\|u_h^*\|_{L^2(\Omega)}}, \quad \textrm{Rel. $H^1$ Semi Err.} := \frac{\|\nabla u_{\textrm{pred}} - \nabla u_h^*\|_{L^2(\Omega)}}{\|\nabla u_h^*\|_{L^2(\Omega)}}.$$
In addition, to measure the error introduced by incorporating neural-network approximation into the original FEM, we also evaluated the accuracy of the FEM solution computed on the same mesh as the solution predicted by FEONet. These values are shown in parentheses in the tables below. Moreover, we measured the memory usage of trainable parameters to evaluate computational efficiency.

Our model was implemented in JAX(v0.4.7), using Flax and Optax libraries. All experiments were conducted on a workstation with a single NVIDIA RTX 3090 GPU(24GB VRAM), running CUDA 11.4 and CUDNN 8.2.4.

\subsection{Advection-diffusion-reaction equation}

For the basic performance evaluation, we first consider the 2D advection-diffusion-reaction equation defined as
\begin{equation}
    \begin{aligned}
     -0.1\ \Delta u(x,y) + a\cdot\nabla u(x,y) + 20\ u(x,y)&= f(x,y) && (x,y) \in \Omega, \\
     u(x,y) &= 0 && (x,y) \in \partial \Omega,
\end{aligned}
\end{equation}

\noindent
where \( \Omega=[-1, 1]^2\), \(a=(-1,0)^T\). Moreover, we set external forces as inputs of neural networks, which are given by
\begin{equation}\label{f_input}
    f(x,y)=m_0\sin(n_0x+n_1y)+m_1\cos(n_2x+n_3y),
\end{equation}
\noindent
where $m_0, \, m_1$ and $n_0, \, n_1, \, n_2, \, n_3$ are random samples from $[0,1)$ and $[0,1)\times\pi$, respectively.

All models consist of five hidden layers. 
We consider a uniform Cartesian grid with $h = 2/n$ in both $x$ and $y$ axis directions, and take the standard right isosceles triangle split yielding a triangulation of $\Omega$ for a conforming piecewise linear finite element method. Corresponding to the mesh resolutions of $n=16,\ 32$, and $64$, the connectivity is $C_{\ell}=1$, $C_{\ell}=2$, and $C_{\ell}=3$ for each respective resolution. 
We determined these connectivity values empirically by testing increasing connectivity levels and selecting those that provide stable training and accurate predictions. To ensure a fair comparison, each model was trained for 10,000 epochs using the same optimizer and learning rate scheduler. 

\begin{table}[t]
\centering 
\caption{Comparison of the baseline FEONet and the sparse FEONet models across three mesh resolutions for the 2D advection-diffusion-reaction equation. The values outside parentheses represent the error between the model prediction and the reference solution, the values inside parentheses represent the error between the finite element solution at resolution $n$ and the reference solution.}
\label{tab:result_adr_table}  

\resizebox{\textwidth}{!}{
\begin{tabular}{@{} c l c c c c c c @{}}
    \toprule
    \textbf{\(n\)} &
    \textbf{Connection} &
    \textbf{\# Params} &
    \textbf{Memory(MB)} &
    \textbf{\%} &
    \textbf{Loss} &
    \textbf{Rel. $L^2$ Err.} &
    \textbf{Rel. $H^1$ Semi Err.} \\ 
    \midrule
    
    \multirow{2}{*}{$16$}
    & Dense & 305{,}100 & 1.2 & 100 & $\mathbf{1.44\times10^{-7}}$ & 0.0676 (0.0675) & \textbf{0.4482} (0.4482) \\
    & \textbf{Sparse ($C_\ell =1$)} & \textbf{10{,}092} & \textbf{0.0404} & \textbf{3.31} & $1.46\times10^{-6}$ & 0.0676 (0.0675) & 0.4483 (0.4482) \\
    \midrule
    
    \multirow{2}{*}{$32$}
    & Dense & 5{,}546{,}892  & 22.2 & 100 & $1.11\times10^{-5}$ & 0.0195 (0.0194) & 0.2519 (0.2519) \\
    & \textbf{Sparse ($C_\ell =2$)} & \textbf{108{,}000} & \textbf{0.4320} & \textbf{1.95} & $\mathbf{1.25\times10^{-7}}$ & \textbf{0.0194} (0.0194) & 0.2519 (0.2519) \\
    \midrule
    
    \multirow{2}{*}{$64$}
    & Dense & 94{,}541{,}580 & 378.2 & 100 & $1.19\times10^{-6}$ & 0.0074 (0.0051) & 0.1314 (0.1304) \\
    & \textbf{Sparse ($C_\ell = 3$)} & \textbf{863{,}088} & \textbf{3.5}  & \textbf{0.91} & $\mathbf{2.81\times10^{-8}}$ & \textbf{0.0052} (0.0051) & \textbf{0.1305} (0.1304) \\
    \bottomrule
    
    \end{tabular}
    }
\end{table}

\begin{figure}[t]
    \centering
    \includegraphics[width=\textwidth]{sparse_c5_adr_64.png}
    \caption{
    Visualization of the 2D advection-diffusion-reaction problem results on mesh resolution \(n=64\): The figure displays the input external forcing term $f(x,y)$, the FEM solution $u_h$, Sparse FEONet prediction $u_{\rm{pred}}$, and the absolute error $|u_{\rm{pred}} - u_h|$. 
    }
    \label{fig:adr_scnn_plot}
\end{figure}

Table \ref{tab:result_adr_table} shows the number of parameters, memory, weak form loss, relative \(L^2\) error, and relative \(H^1\) semi-norm error on the test set. The error and loss values are averaged over 3000 data samples per epoch. The sparse neural network performs comparably to, or even better than the dense network with up to about 99\% fewer parameters. In the case of the coarse grid ($n=16$), the errors of both models are similar to those of the finite element method, which is the value in the parentheses. As the resolution increases, the sparse model shows better convergence. At $n=64$, the dense model achieves 0.0074 in relative $L^2$ error, while the sparse model achieves 0.0052, which is close to the accuracy of the finite element solution at that resolution. Figure \ref{fig:adr_scnn_plot} shows the input function, the finite element solution, the sparse FEONet prediction, and the absolute error between the finite element solution and the prediction in \( n = 64 \). 

\subsection{Helmholtz equation}

Some problems require high-resolution solutions with many basis functions due to stiffness or highly oscillatory behavior. In such regimes, the original FEONet often faces substantial computational difficulties, whereas our proposed method can predict solutions effectively. To illustrate this point, we consider the 2D Helmholtz equation following the formulation in \cite{Mcclenny2023}, given by
\begin{equation}
    \begin{aligned}
        \Delta u(x,y) + k^2 u(x,y) &=q(x,y) 
        && (x,y) \in \Omega, \\
        u(x,y) &= g(x,y) && (x,y) \in \partial \Omega,
    \end{aligned}
\end{equation}

\noindent
where \( \Omega=[-1, 1]^2\). We construct data using the manufactured solution. More precisely, we prescribe the exact solution
\begin{equation}
    u(x,y)=\sin(a_1 \pi x)\sin(a_2 \pi y),
\end{equation}
and obtain the corresponding forcing term $q(x,y)$ and boundary data $g(x,y)$ given by
\begin{equation}
    \begin{aligned}
        q(x,y) 
        &= - (a_1 \pi)^2 \sin(a_1 \pi x)\sin(a_2 \pi y) 
        - (a_2 \pi)^2 \sin(a_1 \pi x)\sin(a_2 \pi y) 
        + k^2 \sin(a_1 \pi x)\sin(a_2 \pi y), \\
        g(x,y) &=u\big|_{\partial \Omega},
    \end{aligned}
\end{equation}
with the random samples \(a_1, \, a_2 \in [1, 10),\) where we fixed $k = 1$.

To make the problem more challenging, we sampled relatively large values of $a_1$, and $a_2$. This produces high oscillations, which require fine-mesh resolutions for accurate solution prediction. This setting allows us to evaluate the model's capability in this high-resolution regime.

\begin{table}[t]
\centering 
\caption{Comparison of the baseline FEONet and the sparse FEONet models across four mesh resolutions for the 2D Helmholtz equation. The values outside parentheses represent the error between the model prediction and the reference solution, the values inside parentheses represent the error between the finite element solution at resolution $n$ and the reference solution.}
\label{tab:result_helmholtz_table}  

\resizebox{\textwidth}{!}{
\begin{tabular}{@{} c l c c c c c c @{}}
    \toprule
    \textbf{\(n\)} &
    \textbf{Connection} &
    \textbf{\# Params} &
    \textbf{Memory(MB)} &
    \textbf{\%} &
    \textbf{Loss} &
    \textbf{Rel. $L^2$ Err.} &
    \textbf{Rel. $H^1$ Semi Err.} \\ 
    \midrule
    
    $16$ & Dense & 305{,}100 & 1.2 & 100 & 54.9847 & 0.9720 (0.7708) & 0.9827 (0.8113) \\
    \midrule
    
    \multirow{2}{*}{$32$}
    & Dense & 5{,}546{,}892  & 22.2 & 100 & 888.0010 & 0.9619 (0.3902) & 1.0022 (0.4994) \\
    & \textbf{Sparse ($C_\ell =5$)} & \textbf{451{,}752} & \textbf{1.80} & \textbf{8.14} & \textbf{0.0787} & \textbf{0.3940} (0.3902) & \textbf{0.5012} (0.4994) \\
    \midrule

    \multirow{2}{*}{$64$}
    & Dense & 94{,}541{,}580 & 378.2 & 100 & 598.9622 & 0.9742 (0.1381) & 1.0560 (0.2650) \\
    & \textbf{Sparse ($C_\ell =7$)} & \textbf{3,635,940} & \textbf{14.56}  & \textbf{3.85} & \textbf{0.0036} & \textbf{0.1411} (0.1381) & \textbf{0.2673} (0.2650)  \\
    \midrule
    
    \multirow{2}{*}{$128$}
    & Dense & 1{,}560{,}964{,}620 & 6{,}200.0 & 100 & - & - & - \\
    & \textbf{Sparse ($C_\ell =10$)} & \textbf{29{,}824{,}248} & \textbf{119.3}  & \textbf{1.91} & \textbf{0.0004} & \textbf{0.0414} (0.0401) & \textbf{0.1343} (0.1334)  \\
    \bottomrule
    
    \end{tabular}
    }
\end{table}

\begin{figure}[t]
    \centering
    \includegraphics[width=\textwidth]{sparse_c10_helmholtz_128.png}
    \caption{
    Visualization of the 2D Helmholtz problem results on mesh resolution \(n=128\): The figure displays the input external forcing term $f(x,y)$, the FEM solution $u_h$, on the same mesh, the sparse FEONet prediction $u_{\rm{pred}}$, and the absolute error $|u_{\rm{pred}} - u_h|$. 
    }
    \label{fig:helmholtz_scnn_plot}
\end{figure}

We conducted the experiments using the same setting as in the advection-diffusion-reaction equation, except that each model was trained for longer epochs because of the increased difficulty caused by the nonhomogeneous Dirichlet boundary conditions. We tested the model on mesh resolutions, $n=16,\ 32,\ 64$, and $128$. The sparse model uses the same local connectivity strategy as before with $C_{\ell}=5$ for $n=32$, $C_{\ell}=7$ for $n=64$, and $C_\ell=10$ for $ n=128$. Because the solution exhibits high oscillations, the coarse grid cannot capture the overall behavior of the solution. Although finer mesh resolutions can resolve these oscillations, the dense model has significant optimization difficulties and fails to converge in this regime. Also, for $n=128$, the number of parameters in the dense model results in the GPU memory overflow, which prevented the experiment from being conducted. However, the sparse model converges well on $n=32, 64$, and even on the high-resolution case $n=128$. On $n=128$, our model achieves stable convergence with only about 1.91\% of the parameters required by the dense model. Table \ref{tab:result_helmholtz_table} summarizes the results, and Figure \ref{fig:helmholtz_scnn_plot} shows the plot of the prediction of the sparse model at $n=128$.

\begin{table}[t!]
\centering 
\caption{Comparison of the baseline FEONet and the sparse FEONet models across three mesh resolutions for the 1D nonlinear Burgers' equation. The values outside parentheses represent the error between the model prediction and the reference solution, the values inside parentheses represent the error between the finite element solution at resolution $n$ and the reference solution.}
\label{tab:result_nb_table}  

\resizebox{\textwidth}{!}{
\begin{tabular}{@{} c l c c c c c c c @{}}
    \toprule
    \textbf{\(n\)} &
    \textbf{Connection} &
    \textbf{\# Params} &
    \textbf{Memory(MB)} &
    \textbf{\%} &
    \textbf{Loss} &
    \textbf{Rel. $L^2$ Err.} &
    \textbf{Rel. $H^1$ Semi Err.} \\ 
    \midrule
    
    \multirow{2}{*}{$64$}
    & Dense & 24{,}192 & 0.0968 & 100 & \textbf{0.00030} & \textbf{0.0072} (0.0039) & 0.1123 (0.1121) \\
    & \textbf{Sparse ($C_\ell =8$)} & \textbf{6{,}372} & \textbf{0.0255} & \textbf{26.34} & 0.00043 & 0.0104 (0.0039) & 0.1123 (0.1121) \\
    \midrule
    
    \multirow{2}{*}{$128$}
    & Dense & 97{,}536 & 0.3901 & 100 & 0.00602 & 0.0102 (0.0010) & 0.0569 (0.0558) \\
    & \textbf{Sparse ($C_\ell =13$)} & \textbf{20{,}244} & \textbf{0.0810} & \textbf{20.76} & \textbf{0.00379} & \textbf{0.0035} (0.0010) & \textbf{0.0561} (0.0558) \\
    \midrule
    
    \multirow{2}{*}{$256$}
    & Dense & 391{,}680 & 1.6 & 100 & 0.04933 & 0.1194 (0.0002) & 0.1930 (0.0273)\\
    & \textbf{Sparse ($C_\ell =30$)} & \textbf{89{,}280} & \textbf{0.3571}  & \textbf{22.79} & \textbf{0.02336} & \textbf{0.0056} (0.0002) & \textbf{0.0288} (0.0273) \\
    \bottomrule
    
\end{tabular}
}
\end{table}

\subsection{Nonlinear Burgers' equation}

We next show that our method also performs well on nonlinear equations. To do this, we tested our model on the 1D nonlinear Burgers' equation to assess the model's capability in handling nonlinear terms. The equation can be written as
\begin{equation}
    \begin{aligned}
     -0.1u_{xx}+uu_x &= f(x) && x \in [-1,1], \\
     u(x) &= 0 && x \in \{-1, 1\},
\end{aligned}
\end{equation}
where $f(x)=m_0\sin(n_0x)+m_1\cos(n_1x)$ with the random samples $m_0, \, m_1 \in [0,1)$ and  $n_0, \, n_1\in [0,1)\times\pi$. For the sparse model, we set the connectivity to $C_\ell=8$ for $n=64$, $C_\ell=13$ for $n=128$, and $C_\ell=30$ for $n=256$. 

\begin{figure}[t]
    \centering
    \includegraphics[width=\textwidth]{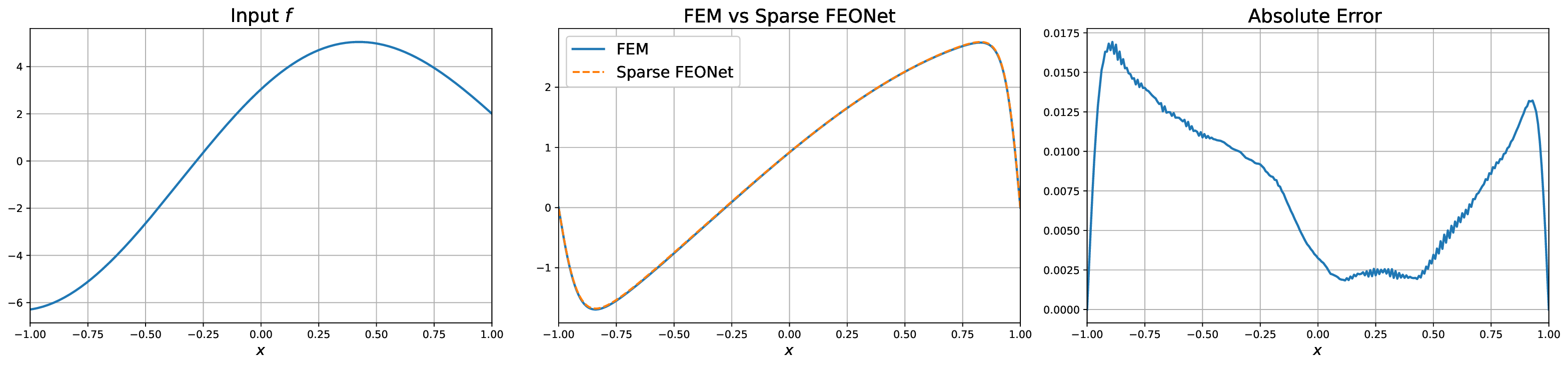}
    \caption{
    Visualization of the results on the 1D nonlinear Burgers' problem on mesh resolution \(n=256\): The figure displays the input forcing term $f(x)$, the FEM solution $u_h$, Sparse FEONet prediction $u_{\rm{pred}}$, and the absolute error $|u_{\rm{pred}} - u_h|$. The mean relative $L^2$ error on the test set is $0.0056$ using $22.79\%$ of the parameters.
    }
    \label{fig:nb_scnn_plot}
\end{figure}

Table \ref{tab:result_nb_table} presents the quantitative results, and Figure \ref{fig:nb_scnn_plot} shows the prediction of the sparse model at $n=256$.
The sparse model achieves convergence in the relative $H^1$ semi-norm error that is comparable to that of the finite element solution across all resolutions.
In contrast, the dense model shows a noticeable difference from the FEM as $n$ increases.
This behavior suggests that the optimization of the dense model becomes more difficult at higher resolutions, and that the training may not be sufficient within the fixed number of epochs.
For the relative $L^2$ error, the sparse model still achieves small errors, but a larger gap with respect to the finite element solution is observed.
This difference is likely due to the weak-form-residual-based training, which directly controls the gradient of the solution and therefore favors accuracy in the $H^1$ semi-norm, while the $L^2$ error is only indirectly controlled.
In addition, the nonlinear and advective nature of the Burgers' equation can make the $L^2$ error more sensitive to small shifts in the solution profile.
Compared to the two-dimensional case, the ratio of parameters between the dense and sparse models is relatively large in one dimension. However, since each node has fewer neighboring degrees of freedom, the overall reduction in the number of parameters becomes less significant. These results indicate that the proposed sparse model remains effective for problems with nonlinear terms.

\subsection{Irregular triangular meshes and complex domain}

To evaluate the performance of the proposed sparse model for a general case, we test the model on irregular triangular meshes and examine whether the sparse connectivity strategy remains effective when the mesh structure is not aligned with a regular Cartesian grid. To do this, we solve the 2D Poisson equation with the homogeneous Dirichlet boundary condition and the forcing terms given by \eqref{f_input}. Using FEniCS, we generate irregular triangular meshes on a standard square domain and on a square domain with a circular hole. The two domains have 919 and 679 degrees of freedom, respectively. Except for the structure of the mesh and the domain geometry, all training settings are kept consistent with the previous experiments.

Table \ref{tab:result_poisson_diff_geometry_table} shows the experimental results. In both cases, the sparse model converges accurately with fewer parameters compared to the dense model. These results indicate that the local connectivity-based sparse structure remains effective on unstructured meshes and complex geometries. Figure \ref{fig:circlehole_poisson_scnn_plot} visualizes the prediction of the sparse model on the domain with a circular hole.

\begin{table}[t]
\centering 
\caption{
Comparison of the baseline FEONet and the sparse FEONet models across different mesh and domain geometries on the 2D Poisson equation. The relative \(L^2\) error and the relative \(H^1\) semi-norm error are computed on the test set using a reference finite element solution at the same mesh resolution.}
\label{tab:result_poisson_diff_geometry_table}  

\resizebox{\textwidth}{!}{
\begin{tabular}{@{} c l c c c c c c @{}}
    \toprule
    \textbf{Domain Type} &
    \textbf{Connection} &
    \textbf{\# Params} &
    \textbf{Memory(MB)} &
    \textbf{\%} &
    \textbf{Loss} &
    \textbf{Rel. $L^2$ Err.} &
    \textbf{Rel. $H^1$ Semi Err.} \\ 
    \midrule
    
    \multirow{2}{*}{Irregular Triangular mesh}
    & Dense & 5{,}072{,}880 & 20.30 & 100 & $3.42\times10^{-6}$ & 0.0011 & 0.0019 \\
    & \textbf{Sparse ($C_\ell =5$)} & \textbf{483{,}300} & \textbf{1.90} & \textbf{9.53} & $\mathbf{5.10\times10^{-7}}$ & \textbf{0.0003} & \textbf{0.0004} \\
    \midrule
    
    \multirow{2}{*}{Circle hole}
    & Dense & 2{,}770{,}320  & 11.1 & 100 & $2.50\times10^{-6}$ & 0.0016 & 0.0025 \\
    & \textbf{Sparse ($C_\ell =3$)} & \textbf{135{,}372} & \textbf{0.5415} & \textbf{4.89} & $\mathbf{2.02\times10^{-7}}$ & \textbf{0.0010} & \textbf{0.0011} \\
    \bottomrule
    
    \end{tabular}
    }
\end{table}

\begin{figure}[t]
    \centering
    \includegraphics[width=\textwidth]{sparse_c3_poisson_circlehole.png}
    \caption{
    Visualization of the 2D Poisson problem on circle-hole domain: External forcing term $f(x,y)$ for input, FEM solution $u_h$ on the same mesh, sparse FEONet prediction $u_{\rm{pred}}$, and the absolute error $|u_{\rm{pred}} - u_h|$. 
    }
    \label{fig:circlehole_poisson_scnn_plot}
\end{figure}

\begin{table}[t]
\centering 
\caption{Results of the sparse FEONet models across three mesh resolutions for the 2D Stokes equation. The values outside parentheses represent the error between the model prediction and the reference solution, the values inside parentheses represent the error between the finite element solution at resolution $n$ and the reference solution.}
\label{tab:result_stokes_table}  

\resizebox{\textwidth}{!}{
\begin{tabular}{@{} c l c c c c c c c @{}}
    \toprule
    \textbf{\(n\)} &
    \textbf{Connection} &
    \textbf{\# Params} &
    \textbf{Memory(MB)} &
    \textbf{\%} &
    \textbf{Loss} &
    \textbf{Rel. $L^2$ Err. of $p$} &
    \textbf{Rel. $H^1$ Semi Err. of $\mathbf{u}$} \\ 
    \midrule
    
    $8$ & Sparse ($C_\ell =3$) & 541{,}548 & 2.2 & 31.83 & $2.25\times10^{-6}$ & 0.0321 (0.0164) & 0.0955 (0.0751) \\
    \midrule
    
    $16$ & Sparse ($C_\ell =5$) & 6{,}708{,}756 & 26.8 & 22.84 & $1.35\times10^{-5}$ & 0.0153 (0.0044) & 0.0311 (0.0210) \\
    \midrule
    
    $32$ & Sparse ($C_\ell =6$) & 44{,}889{,}984 & 179.6 & 9.18 & $2.49\times10^{-6}$ & 0.0530 (0.0012) & 0.0647 (0.0058) \\
    \bottomrule
    
\end{tabular}
}
\end{table}

\subsection{Stokes equations}

Finally, we extend our evaluation to the 2D Stokes equations:
\begin{equation}
    \begin{aligned}
         -\Delta \mathbf{u}(x,y) + \nabla p(x,y) &= \mathbf{f}(x,y) && (x,y) \in \Omega, \\
         \nabla \cdot \mathbf{u}(x,y) &= 0 && (x,y) \in \Omega, \\
         \mathbf{u}(x,y) &= \mathbf{0} && (x,y) \in \partial \Omega,
    \end{aligned}
\end{equation}
where \( \Omega=[-1, 1]^2\), \( \mathbf{u}=(u_1, u_2)\) represents the velocity field, and $p$ denotes the kinematic pressure. 
To ensure uniqueness of the pressure, we also impose the zero-mean condition, $\int_\Omega p\ {\rm{d}}x = 0$.
The forcing term $\mathbf{f}=(f_1, f_2)$ is of the form
\begin{equation}
    f_i(x,y)=m_{i,0}\sin(n_{i,0}x+n_{i,1}y)+m_{i,1}\cos(n_{i,2}x+n_{i,3}y),
\end{equation}
\noindent
with the random samples \(m_{i,0}, \, m_{i,1} \in [1,2],\, n_{i,0}, n_{i,1}, n_{i,2}, n_{i,3} \in [0,1)\times\pi \) for \(i=1,2\).

We use $P_2$ elements for the velocity components and $P_1$ elements for the pressure field. This significantly increases the total number of degrees of freedom compared to the previous problems, making the learning process much more challenging. The discrete finite element system consists of three components related to the momentum equation $-\Delta \mathbf{u} + \nabla p = \mathbf{f}$, the incompressibility condition $\nabla\cdot\mathbf{u} = 0$, and the zero-mean pressure condition $\int_\Omega p\ {\rm{d}}x = 0$.
We define the training loss as a weighted sum of the $l_2$ norms of the residuals of these three components.
We use the constant weights to balance the scales of the residual terms.

We conducted the experiments on mesh resolutions of $n=8, 16$ and $32$. Due to the high computational and memory requirements for this problem, training a dense model becomes impractical. Therefore, we focus on evaluating whether the proposed sparse FEONet converges well in this problem. The results are summarized in Table \ref{tab:result_stokes_table}. The sparse model shows high computational efficiency. At $n=32$, it requires only $9.18 \%$ of the parameters relative to a dense model. At $n=16$, the relative $H^1$ semi-norm error for the velocity field is $0.0311$, which is close to the error between the reference solution at $n=128$ and the finite element solution at $n=16$. At the finer resolution, $n=32$, the problem becomes more challenging due to the rapid increase in the number of degrees of freedom and physical constraints. To train the model for this problem, we reduced the learning rate and extended the training epochs. Although this improves convergence, the prediction errors increase moderately compared with those on the coarse mesh. Nevertheless, the sparse model still converges reliably and achieves acceptable accuracy under this challenging setting. The same model architecture is used, and the weight settings are the same across the mesh resolutions without additional hyperparameter tuning. Figure \ref{fig:stokes_scnn_plot} visualizes the prediction of the sparse model at $n=32$.

\begin{figure}[t]
    \centering
    \includegraphics[width=\textwidth]{sparse_c6_stokes_32.png}
    \caption{
    Visualization of the 2D Stokes problem on mesh resolution \(n=32\): External forcing term $f(x,y)$ for input, FEM solutions $\mathbf{u}_h$ and $p_h$ on the same mesh, the corresponding sparse FEONet predictions, and the absolute errors for the velocity components. 
    }
    \label{fig:stokes_scnn_plot}
\end{figure}

\section{Concluding Remarks}\label{sec:conclusion}
In this work, we consider the FEONet, an unsupervised operator-learning framework for parametric PDEs based on the classical FEM. While FEONet demonstrates strong accuracy and robustness over a wide range of problems, its computational burden grows rapidly with mesh refinement, and its performance can degrade as the number of elements increases, which limits its applicability to large-scale settings. To overcome these challenges, we proposed a new sparse network architecture guided by the intrinsic locality and connectivity structure of finite elements. The proposed design significantly reduces the computational cost and improves memory efficiency, while preserving accuracy comparable to the original FEONet across extensive numerical experiments. Beyond empirical validation, we established theoretical guarantees: we proved that the sparse architecture can approximate the target operator effectively, and we provided a stability analysis that supports reliable training and prediction. Taken together, these results suggest that incorporating finite-element structure into operator networks is a principled and practical strategy for scaling FEONet to finer discretizations.

Looking ahead, an important future direction is to deepen the theoretical understanding of the proposed method by exploiting properties that are specific to the sparse FEONet and are not present in the original FEONet. In particular, by leveraging the finite-element-induced sparsity pattern and its associated locality structure, it should be possible to carry out a convergence analysis that more directly reflects the role of the new architecture, clarifying how the sparse connectivity affects approximation error and generalization error as the mesh is refined. Such a result would not only strengthen the theoretical foundation of the method but also provide principled guidelines for designing and tuning sparse architectures in large-scale regimes.
Another promising avenue is to move beyond multi-layer perceptrons and develop appropriate sparse designs for more general neural network classes. Constructing FEM-relevant sparse structures for broader architectures may significantly widen the applicability of the approach. Establishing conditions under which such sparse generalizations preserve approximation power and stability, while maintaining computational advantages, would be an important and intriguing step toward making the proposed framework a more universal tool for operator learning in parametric PDEs.

\bibliography{references}
\bibliographystyle{abbrv}

\renewcommand{\thetable}{A.\arabic{table}}
\setcounter{table}{0}

\section*{Appendix}

In this appendix section, we investigate the effect of connectivity $C_\ell$ for both the advection-diffusion-reaction equation and the Helmholtz equation.
In all experiments, we consider a uniform mesh on the unit square domain described in Section~\ref{sec:method}, and we report results for mesh resolutions $n = 32$ and $n = 64$.
The number of hidden layers is fixed to five throughout all experiments.
As illustrated in Figure~\ref{fig:propagation}, information from a single node can propagate across the entire domain through successive layers, even with local connectivity.
Based on this observation, one can estimate the minimum value of $C_\ell$ required for global information propagation.
The resulting values are $C_\ell=3$ for $n=16$, $C_\ell=5$ for $n=32$, $C_\ell=11$ for $n=64$, and $C_\ell=21$ for $n=128$.

Tables~\ref{tab:adr_connection_test_32}--\ref{tab:helm_connection_test_64} summarize the effect of the connectivity parameter $C_\ell$ on the performance of the sparse model.
Tables~\ref{tab:adr_connection_test_32} and \ref{tab:adr_connection_test_64} show the results for the advection-diffusion-reaction equation at mesh resolutions $n = 32$ and $n = 64$, respectively, while Tables~\ref{tab:helm_connection_test_32} and \ref{tab:helm_connection_test_64} present the corresponding results for the Helmholtz equation.
In all cases, the relative errors are evaluated on the test set using the finite element solution $u_h$ computed on the same mesh.
The results show that increasing the connectivity generally improves the accuracy of the sparse model, although accurate approximations are already achieved at connectivity levels below the corresponding estimated minimum connectivity.
In addition, the Helmholtz equation, which involves higher-frequency solution components, requires larger connectivity values to achieve comparable accuracy.
These results indicate that the minimum connectivity required for accurate approximation increases with both problem complexity and mesh resolution.

\begin{table}[t]
\centering 
\caption{Convergence of the relative error of the sparse model with varying connectivity $C_\ell$ for the advection-diffusion-reaction equation at mesh resolution $n=32$. The relative error is evaluated on the test set using the finite element solution on the same mesh. The theoretical minimum connectivity for global information propagation is $C_\ell=5$.}
\label{tab:adr_connection_test_32}  

\begin{tabular*}{0.8\textwidth}{@{\extracolsep{\fill}} c c c c @{}}
    \toprule
    \textbf{Connection ($C_\ell$)} &
    \textbf{\# Params} &
    \textbf{\%} &
    \textbf{Rel. Err. of Test set} \\ 
    \midrule
    
    Dense & 5{,}546{,}892 & 100 & 0.0011 \\
    \midrule
    
    $C_\ell=1$ & 44{,}652 & 0.80 & 0.0019 \\
    $C_\ell=2$ & 108{,}000 & 1.95 & 0.0007 \\
    $C_\ell=3$ & 198{,}768 & 3.58 & 0.0006 \\
    $C_\ell=4$ & 314{,}232 & 5.67 & 0.0006 \\
    $C_\ell=5$ & 451{,}752 & 8.14 & 0.0005 \\
    \bottomrule
    
\end{tabular*}
\end{table}

\begin{table}[t]
\centering 
\caption{Convergence of the relative error of the sparse model with varying connectivity $C_\ell$ for the advection-diffusion-reaction equation at mesh resolution $n=64$. The relative error is evaluated on the test set using the finite element solution on the same mesh. The theoretical minimum connectivity for global information propagation is $C_\ell=11$.}
\label{tab:adr_connection_test_64}

\begin{tabular*}{0.8\textwidth}{@{\extracolsep{\fill}} c c c c @{}}
    \toprule
    \textbf{Connection ($C_\ell$)} &
    \textbf{\# Params} &
    \textbf{\%} &
    \textbf{Rel. Err. of Test set} \\ 
    \midrule
    
    Dense & 94{,}541{,}580 & 100 & 0.0052 \\
    \midrule
    
    $C_\ell=1$ & 187{,}500 & 0.20 & 0.0694 \\
    $C_\ell=2$ & 461{,}280 & 0.49 & 0.0023 \\
    $C_\ell=3$ & 863{,}088 & 0.91 & 0.0011 \\
    $C_\ell=4$ & 1{,}387{,}128 & 1.47 & 0.0009 \\
    $C_\ell=5$ & 2{,}027{,}688 & 2.14 & 0.0008 \\
    $C_\ell=6$ & 2{,}779{,}140 & 2.94 & 0.0008 \\
    $C_\ell=7$ & 3{,}635{,}940 & 3.85 & 0.0007 \\
    $C_\ell=8$ & 4{,}592{,}628 & 4.86 & 0.0005 \\
    $C_\ell=9$ & 5{,}643{,}828 & 5.97 & 0.0005 \\
    $C_\ell=10$ & 6{,}784{,}248 & 7.18 & 0.0006 \\
    $C_\ell=11$ & 8{,}008{,}680 & 8.47 & 0.0006 \\
    \bottomrule
    
\end{tabular*}
\end{table}

\begin{table}[t]
\centering 
\caption{Convergence of the relative error of the sparse model with varying connectivity $C_\ell$ for the Helmholtz equation at mesh resolution $n=32$. The relative error is evaluated on the test set using the finite element solution on the same mesh. The theoretical minimum connectivity for global information propagation is $C_\ell=5$.}
\label{tab:helm_connection_test_32}  

\begin{tabular*}{0.8\textwidth}{@{\extracolsep{\fill}} c c c c @{}}
    \toprule
    \textbf{Connection ($C_\ell$)} &
    \textbf{\# Params} &
    \textbf{\%} &
    \textbf{Rel. Err. of Test set} \\ 
    \midrule
    
    Dense & 5{,}546{,}892 & 100 & 0.8993 \\
    \midrule
    
    $C_\ell=1$ & 44{,}652 & 0.80 & 0.3252 \\
    $C_\ell=2$ & 108{,}000 & 1.95 & 0.1434 \\
    $C_\ell=3$ & 198{,}768 & 3.58 & 0.0914 \\
    $C_\ell=4$ & 314{,}232 & 5.67 & 0.0625 \\
    $C_\ell=5$ & 451{,}752 & 8.14 & 0.0440 \\
    \bottomrule
    
\end{tabular*}
\end{table}

\begin{table}[t]
\centering 
\caption{Convergence of the relative error of the sparse model with varying connectivity $C_\ell$ for the Helmholtz equation at mesh resolution $n=64$. The relative error is evaluated on the test set using the finite element solution on the same mesh. The theoretical minimum connectivity for global information propagation is $C_\ell=11$.}
\label{tab:helm_connection_test_64}

\begin{tabular*}{0.8\textwidth}{@{\extracolsep{\fill}} c c c c @{}}
    \toprule
    \textbf{Connection ($C_\ell$)} &
    \textbf{\# Params} &
    \textbf{\%} &
    \textbf{Rel. Err. of Test set} \\ 
    \midrule
    
    Dense & 94{,}541{,}580 & 100 & 0.9585 \\
    \midrule
    
    $C_\ell=1$ & 187{,}500 & 0.20 & 0.1554 \\
    $C_\ell=2$ & 461{,}280 & 0.49 & 0.0519 \\
    $C_\ell=3$ & 863{,}088 & 0.91 & 0.0266 \\
    $C_\ell=4$ & 1{,}387{,}128 & 1.47 & 0.0290 \\
    $C_\ell=5$ & 2{,}027{,}688 & 2.14 & 0.0247 \\
    $C_\ell=6$ & 2{,}779{,}140 & 2.94 & 0.0220 \\
    $C_\ell=7$ & 3{,}635{,}940 & 3.85 & 0.0209 \\
    $C_\ell=8$ & 4{,}592{,}628 & 4.86 & 0.0221 \\
    $C_\ell=9$ & 5{,}643{,}828 & 5.97 & 0.0239 \\
    $C_\ell=10$ & 6{,}784{,}248 & 7.18 & 0.0253 \\
    $C_\ell=11$ & 8{,}008{,}680 & 8.47 & 0.0268 \\
    \bottomrule
    
\end{tabular*}
\end{table}

\end{document}